\documentclass[article,onefignum,onetabnum]{siamart250211}
\usepackage{_settings}
\tikzstyle{process} = [rectangle, minimum width=5cm, text width=5cm, minimum height=1cm, text centered, draw=black, fill=blue!10]
\tikzstyle{startstop} = [rectangle, rounded corners, minimum width=5cm, text width=5cm, minimum height=1cm, text centered, draw=black, fill=gray!20]

\tikzstyle{arrow} = [thick,->,>=stealth]
\usepackage{changes}
\definechangesauthor[name=Victor, color=gray]{lmf}
\definechangesauthor[name=Houssem, color=blue]{pm}
\title{A neural operator framework for solving inverse scattering problems\thanks{Submitted to the editors DATE.}}

\headers{A neural operator framework for solving inverse scattering problems}{V. Chenu, H. Haddar, H. Montanelli}

\ifpdf
\hypersetup{
  pdftitle={tbd},
  pdfauthor={V. Chenu, H. Haddar, H. Montanelli}
}
\fi

\author{Victor Chenu\thanks{Inria, ENSTA Paris, UMA, Institut Polytechnique de Paris, 91120 Palaiseau, France.}
\and Houssem Haddar\footnotemark[2]
\and Hadrien Montanelli\footnotemark[2]}

\begin{document}

\maketitle

\begin{abstract}
We present a neural operator framework for solving inverse scattering problems.
A neural operator produces a preliminary indicator function for the scatterer, which, after appropriate rescaling, is used as a regularization parameter within the Linear Sampling Method to validate the initial reconstruction.
The neural operator is implemented as a DeepONet with a fixed radial-basis-function trunk, while the noise level required for rescaling is estimated using a dedicated neural network.
A neural tangent kernel analysis guides the architectural design, reducing the network tuning to a single discretization parameter, adjustable according to the wavelength.
Two-dimensional numerical experiments demonstrate the method's effectiveness, with a Python toolbox provided for reproducibility.
\end{abstract}

\begin{keywords}
inverse acoustic scattering, Tikhonov regularization, Linear Sampling Method, neural networks, neural operators, DeepONets
\end{keywords}

\section{Introduction}

Inverse scattering problems aim at determining the properties of a medium from its response to one or several incident waves. 
Such problems arise in numerous practical applications, including non-destructive testing and medical imaging. 
A wide range of reconstruction techniques has been developed to tackle them~\cite{garnier2015,borcea2002,colton_kress}.  
We focus here on fixed-frequency approaches using multistatic data, in particular sampling methods~\cite{colton2003linear,ito2012}. 
More precisely, we consider the Linear Sampling Method~(LSM), originally introduced by Colton and Kirsch in~\cite{colton1996}. 
The LSM is typically regularized via Tikhonov regularization, with the parameter selected according to Morozov's discrepancy principle.
Although this strategy generally provides reliable reconstructions, it suffers from two main drawbacks in our setting. 
First, it requires prior knowledge of the noise level $\delta$, which is rarely available in practice. 
Second, determining the regularization parameter $\alpha(\bs{z})$ entails solving a nonlinear equation at each sampling point $\bs{z}$, resulting in a significant computational burden.  
This choice, however, guarantees robustness with respect to the noise level as well as to the size and shape of the obstacle. 
In contrast, it has been observed~\cite{catapano2007} that using a constant regularization parameter---independent of both $\delta$ and $\bs{z}$---can still yield acceptable reconstructions at a substantially lower computational cost, albeit with reduced robustness.
Our goal is therefore to strike a balance between robustness and computational efficiency.

With the rapid development of Scientific Machine Learning (SciML)~\cite{montanelli2025} and its demonstrated success in solving forward problems~\cite{karniadakis2021,li2020,raissi2019}, increasing attention has been devoted to its application to inverse problems~\cite{habring2024,haltmeier2023,molinaro2023}. 
In particular, several recent works have investigated the use of deep neural networks for inverse scattering problems~\cite{lin2024,pourahmadian2025,zhang2025,zhou2023}.  
The proposed strategies range from sophisticated physics-informed architectures designed to solve the full inverse problem end-to-end, to more lightweight networks that either provide an initial guess or enhance reconstructions obtained by classical methods. 
The present work adopts the latter perspective and relies on the neural operator framework~\cite{azizzadenesheli2024,kovachki2023}, which generalizes standard neural networks by enabling the approximation of mappings between function spaces.

We introduce a data-driven neural operator approach that replaces Morozov's discrepancy principle, addressing the limitations discussed above. 
Since Morozov's regularization parameters often resemble scaled indicator functions of the obstacle, we train a neural operator to predict a normalized indicator function capturing its position and shape. 
The appropriate scaling factor is then estimated from the noise level, which is predicted by a secondary neural network. 
The rescaled indicator is subsequently incorporated into the LSM, allowing the method to confirm or refine the initial neural network prediction. 
In this way, our hybrid approach combines data-driven learning with classical inversion: it first generates a neural-network-based indicator and then integrates it into the LSM for validation and enhanced reconstruction.

The remainder of this paper is organized as follows. 
In \cref{section_lsm}, we introduce the forward and inverse scattering problems and present numerical results obtained with the standard LSM combined with Morozov's discrepancy principle. 
\Cref{section_NN_ind} describes the proposed neural operator framework for constructing an initial indicator function of the obstacle. 
\Cref{section_noise} and \cref{section_reg} detail the regularization strategy, covering noise-level estimation and its incorporation into the regularization function, and present the corresponding numerical results, with a particular comparison between the neural and LSM indicators.

\section{Solving the inverse problem with the LSM}\label{section_lsm}

In this section, we introduce the forward and inverse scattering problems, and review the LSM. 

\subsection{Forward and inverse scattering problems}

\paragraph{Forward problem}
We consider the propagation of acoustic waves in a two-dimensional medium containing a sound-soft obstacle $D$. 
Let $u^i$ denote an incident wave, solution to the Helmholtz equation $\Delta u^i + k^2 u^i = 0$ in $\R^2$, where $k$ is the wavenumber. 
The incident field $u^i$ gives rise to a scattered field $u^s$ that solves
\begin{align}
\left\{
\begin{array}{l}
\Delta u^s + k^2 u^s = 0 \quad \text{in } \mathbb{R}^2 \setminus \overline{D}, \\[1mm]
u^s = -u^i \quad \text{on } \partial D, \\[1mm]
u^s \text{ is radiating.}
\end{array}
\right.
\label{pb_dirichlet}
\end{align}
The (Sommerfeld) radiation condition reads
\[
\lim_{r \to \infty} \sqrt{r} \left( \frac{\partial u^s}{\partial r} - i k u^s \right) = 0, \quad r = |x|, \quad \text{uniformly in } \hat{x} = x/|x|.
\]
The forward scattering problem is linear and well-posed. In particular, there is a unique solution $u^s \in H^1_{\mathrm{loc}}(\mathbb{R}^2 \setminus D)$ \cite[Thm.~3.11]{colton_kress}.
Furthermore, the scattered field admits the asymptotic expansion
\begin{align}
u^s(x) = \frac{e^{ik|x|}}{\sqrt{|x|}} \Big( u_\infty(\hat{x}) + \mathcal{O}(|x|^{-1}) \Big),
\quad |x| \to \infty,
\label{ff_expansion}
\end{align}
where $u_\infty: \mathbb{S}^1 \to \mathbb{C}$ is the far-field pattern.

\paragraph{Inverse problem}
We focus on the inverse scattering problem, which consists in recovering the geometry of the obstacle, namely $\partial D$ in \cref{pb_dirichlet}, from measurements of the far-field pattern $u_\infty$. The inverse scattering problem is nonlinear and severely ill-posed.

\subsection{LSM with Tikhonov--Morozov regularization}

The LSM was introduced in 1996 by Colton and Kirsch \cite{colton1996}. 
It transforms the nonlinear inverse scattering problem into a family of linear problems.
The LSM is closely related to the factorization method \cite{kirsch2007} and has been extended to several variants, including the Generalized LSM (GLSM) \cite{audibert2014} and passive imaging \cite{garnier2023, garnier2024}.

\paragraph{Methodology}
Consider a probing domain $\Omega \subset \mathbb{R}^2$ and incident plane waves,
\[
u^i(x,\hat{d}) = e^{ik x \cdot \hat{d}}, 
\qquad x\in \mathbb{R}^2, \ \hat{d} \in \mathbb{S}^1,
\] 
with associated scattered field $u^s(x,\hat{d})$ and far-field pattern $u_\infty(\hat{x},\hat{d})$.
Let $\hat{d}_1,\dots,\hat{d}_n$ denote $n$ equispaced incident directions (emitters) and $\hat{x}_1,\dots,\hat{x}_m$ denote $m$ equispaced measurement directions (sensors); see \cref{lsm_setup}. The measurements are collected in the far-field matrix $F \in \mathbb{C}^{m\times n}$:
\[
F_{ij} = u_\infty(\hat{x}_i, \hat{d}_j), \qquad 1\le i \le m, \ 1\le j \le n.
\]
The LSM constructs an indicator function $I:\Omega \to \mathbb{R}$ of the defect as follows.
For each sampling point $z \in \Omega$, consider the linear system
\begin{align}
F g_{z} = \phi_{z}, 
\qquad (\phi_{z})_i = \frac{e^{i\pi/4}}{\sqrt{8\pi k}} e^{-ik \hat{x}_i \cdot z},
\label{non_regularized_eq}
\end{align}
where $g_{z} \in \mathbb{C}^n$ is unknown and $\phi_{z} \in \mathbb{C}^m$ corresponds to the far-field of the fundamental solution of the Helmholtz equation. 
The indicator is then defined as 
\[
I(z) := \| g_{z} \|_2^{-1} \qquad \text{(LSM indicator)}.
\]

\begin{figure}[t]
\centering
\begin{tikzpicture}[scale=0.8]
\draw[black, thick, fill=black!10] (0,0) ellipse (1cm and 0.5cm);
\node[anchor=north] at (0, 0.2) {$D$};
\draw[black, dashed] (0,0) circle (2.2);
\draw[draw=black!30] (-4,-3) rectangle ++(8,6);
\node at (-3.5,2.5) {$\Omega$};
\node[draw, black, fill=black, cross, thick, scale=0.55, rotate=45] at (2.2*0.7071, 2.2*0.7071) {};
\node[draw, black, fill=black, cross, thick, scale=0.55, rotate=45] at (2.2*0.7071, -2.2*0.7071) {};
\node[draw, black, fill=black, cross, thick, scale=0.55, rotate=45] at (2.2*0.9239,2.2*0.3827) {};
\node[draw, black, fill=black, cross, thick, scale=0.55, rotate=45] at (2.2*0.9239,-2.2*0.3827) {};
\node[draw, black, fill=black, cross, thick, scale=0.55, rotate=45] at (2.2,0) {};
\node at (2.2+0.5,0) {$\hat{x}_i$};
\node[draw, black, fill=black, circle, thick, scale=0.55] at (-2.2*0.7071, 2.2*0.7071) {};
\node[draw, black, fill=black, circle, thick, scale=0.55] at (-2.2*0.7071, -2.2*0.7071) {};
\node[draw, black, fill=black, circle, thick, scale=0.55] at (-2.2*0.9239,2.2*0.3827) {};
\node[draw, black, fill=black, circle, thick, scale=0.55] at (-2.2*0.9239,-2.2*0.3827) {};
\node[draw, black, fill=black, circle, thick, scale=0.55] at (-2.2,0) {};
\node at (-2.2-0.5,0) {$\hat{d}_j$};
\end{tikzpicture}
\caption{Example setup for the LSM: the probing domain $\Omega$, the defect $D$, sources $\hat{d}_j$, and sensors $\hat{x}_i$. In the multistatic configuration considered here, multiple incident waves are emitted from different directions $\hat{d}_j$ and the scattered field is measured at multiple receiver locations $\hat{x}_i$, yielding a full matrix of measurements.}
\label{lsm_setup}
\end{figure}
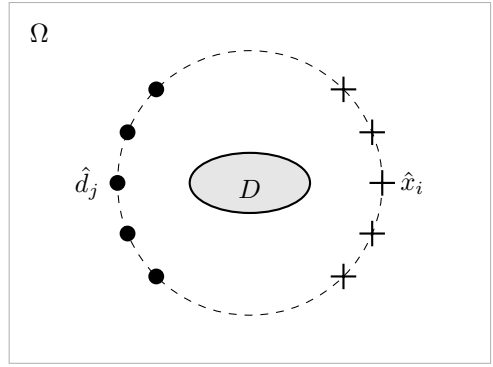

\paragraph{Tikhonov--Morozov regularization}
The system \eqref{non_regularized_eq} is ill-posed, and in practice only a noisy matrix $F_\delta$ is available, with noise level $\delta := \| F - F_\delta \|_2$. 
It is therefore Tikhonov-regularized,
\begin{align}
F_\delta^* F_\delta g_{z} + \alpha_{z} g_{z} = F_\delta^* \phi_{z},
\label{regularized_eq}
\end{align}
where $\alpha_{z} > 0$ is the regularization parameter, chosen via Morozov's discrepancy principle:
\begin{align}
\| F_\delta g_{z} - \phi_{z} \|_2 = \delta \, \| g_{z} \|_2.
\label{morozov}
\end{align}
The parameter $\alpha_{z}$ is obtained by solving a nonlinear equation involving the SVD of $F_\delta$~\cite[Sec.~5]{garnier2023}.

\subsection{Practical LSM setup}\label{LSM_setup}

\begin{figure}[t]
\centering
\def\scl{0.39}
\begin{tabular}{cc}
\text{Morozov regularizer $\alpha(z)$} & \text{LSM indicator $I(z)$} \\
\includegraphics[scale=\scl]{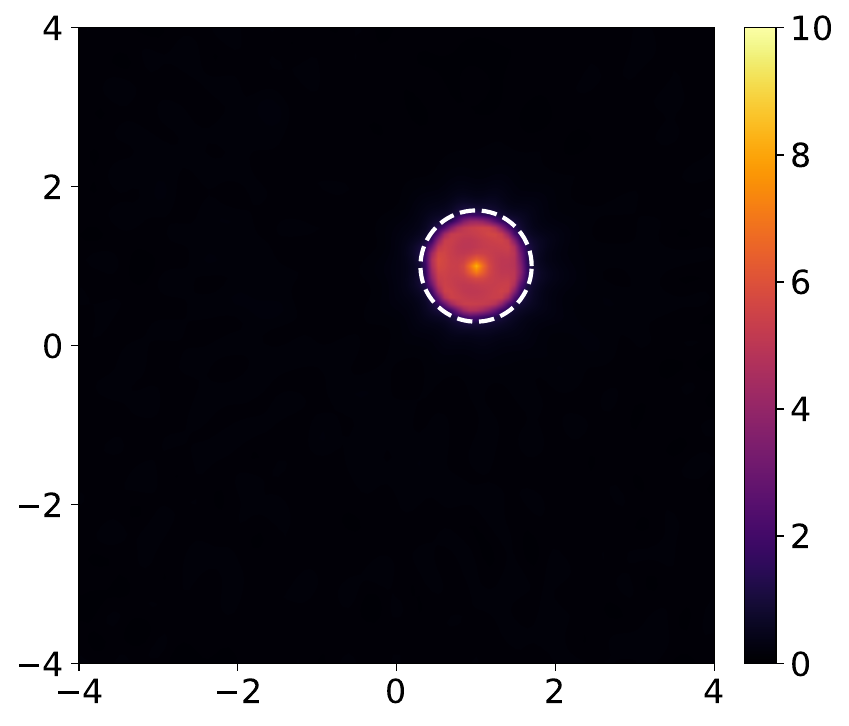} & 
\includegraphics[scale=\scl]{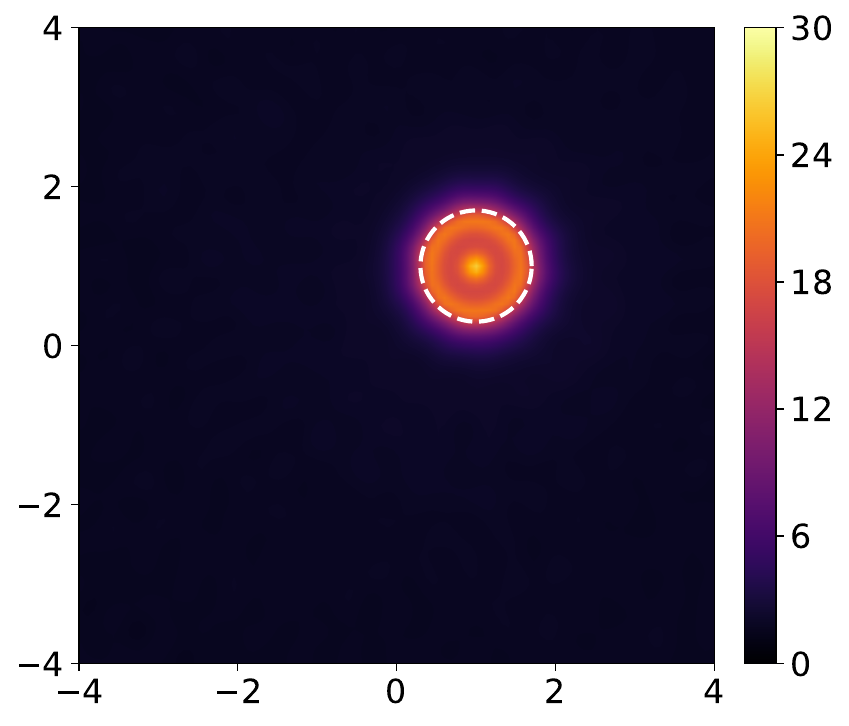} \\
\end{tabular}
\caption{Morozov regularizer (left) and corresponding LSM indicator function (right) for $m=n=30$. 
The Morozov regularization function computed from \cref{morozov} is used in the Tikhonov-regularized LSM linear system \cref{regularized_eq}. 
This is the standard LSM workflow, which we aim to improve in this paper.}
\label{fig_morozov}
\end{figure}

\paragraph{Sources and sensors}
Throughout this paper, we adopt the following setup for numerical experiments. 
We consider $m$ equispaced measurement directions and $n$ equispaced sources,
\[
\hat{x}_i = \bigl(\cos \theta_i, \, \sin \theta_i\bigr), \quad
\hat{d}_j = \bigl(\cos \phi_j, \, \sin \phi_j\bigr),
\]
with
\[
\theta_i = 2\pi \frac{i-1}{n}, \quad
\phi_j = 2\pi \frac{j-1}{m}, \quad
1 \le i \le n, \quad
1 \le j \le m.
\]
The sampling domain is set to $\Omega = [-\lambda L, \lambda L]^2$ with $L=4$, discretized on a $100\times 100$ uniform grid. 
The wavenumber is chosen as $k = 2\pi$, yielding a wavelength $\lambda = 1$.

\paragraph{Solving the forward problem}
For a given medium containing a defect, assembling the far-field matrix $F$ requires solving one direct problem for each incident wave. 
For circular obstacles, analytical solutions are available (see \cref{appendix_a}). 
For more general geometries, we use a Nystr\"{o}m method~\cite{atkinson1992} applied to the boundary integral formulation of the problem. 
For more complex or non-smooth obstacles, boundary element methods provide a suitable alternative~\cite{montanelli2025b,montanelli2022,montanelli2024a}.

\paragraph{Noise}
Once the matrix $F$ is computed, we corrupt it with multiplicative Gaussian noise:
\begin{align}
F_\delta = F + \eta \, F \odot (X + i \, Y),
\label{noising}
\end{align}
where $X$ and $Y$ are $m\times n$ matrices with i.i.d.\ standard normal entries, 
$\odot$ denotes componentwise (Hadamard) multiplication, and $\eta > 0$ controls the noise level $\delta$.\footnote{
Since $\mathbb{E}[\|F_\delta - F\|_2^2] \le \mathbb{E}[\|F_\delta - F\|_F^2] = 2\eta^2 \|F\|_F^2$, 
the absolute noise level $\delta$ is proportional to $\eta$ in mean square.
}

\paragraph{Computations}
All computations are performed using our Python toolbox. 
Some LSM reconstructions, together with the corresponding regularization functions, are shown in \cref{fig_morozov}. 
In the remainder of the paper, we aim to eliminate the need for prior knowledge of the noise level and to replace Morozov's discrepancy principle with a neural operator approach.

\subsection{Our hybrid LSM approach}

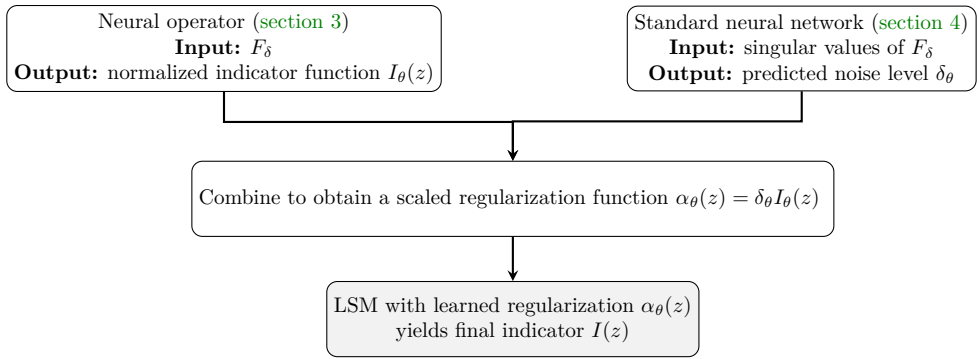
\begin{figure}[t]
\centering

\begin{tikzpicture}[node distance=2cm, every node/.style={align=center}, scale=0.8, transform shape]

\tikzstyle{process} = [rectangle, draw, rounded corners, minimum width=4cm, minimum height=1.25cm, text centered]
\tikzstyle{startstop} = [rectangle, draw, rounded corners, minimum width=4cm, minimum height=1.25cm, text centered, fill=gray!10]]
\tikzstyle{arrow} = [thick,->,>=stealth]

\node (noise) [process] at (4.8,0) {Standard neural network (\cref{section_noise}) \\ \textbf{Input:} singular values of $F_{\delta}$ \\ \textbf{Output:} predicted noise level $\delta_\theta$};
\node (DeepONet) [process] at (-4.8,0) {Neural operator (\cref{section_NN_ind}) \\ \textbf{Input:} $F_{\delta}$ \\ \textbf{Output:} normalized indicator function $I_\theta(\bs{z})$ };

\node (combine) [process] at (0,-2.5) {Combine to obtain a scaled regularization function $\alpha_\theta(z) = \delta_\theta   I_\theta(z)$
};
\node (lsm) [startstop, below of=combine, yshift=0.0cm] {LSM with learned regularization $\alpha_\theta(z)$\\ yields final indicator $I(z)$};

\draw [arrow] (noise.south) -- ++(0,-0.5) -| (combine.north);
\draw [arrow] (DeepONet.south) -- ++(0,-0.5) -| (combine.north);
\draw [arrow] (combine.south) -- ++(0,0) -- (lsm.north);

\end{tikzpicture}
\label{diagram}
\caption{Schematic overview of the proposed methodology: the neural operator and the noise prediction network are trained independently; their outputs are combined to construct a regularization function, which is then used within the LSM.}
\end{figure}

The workflow is summarized in \cref{diagram}.  
We propose a hybrid SciML approach that combines the LSM with two independently trained neural networks. 
The first, implemented as a neural operator, maps the noisy far-field matrix $F_\delta$ to a normalized indicator function $I_\theta(z)$ encoding geometric information on the location and size of the scatterer.  
The second network predicts the noise level $\delta_\theta$ directly from the singular values of $F_\delta$.
For a given noisy far-field matrix, the two outputs are combined to define a regularization function
\begin{align}
\alpha_\theta[F_\delta](z) = \delta_\theta[F_\delta]\,
I_\theta[F_\delta](z),
\label{final_alpha_nn}
\end{align}
which incorporates both the estimated noise magnitude and prior geometric information on the obstacle.  
This learned regularization is then used within the LSM to compute the final indicator function $I(\bs{z})$, allowing for a direct comparison with the initial neural-network-based indicator.


This work illustrates the core philosophy of hybrid SciML: neural networks are not used as a replacement for classical inversion techniques, but rather as a means to enhance them by injecting learned information into theoretically grounded algorithms.  
To keep the approach practical and reproducible, we deliberately adopt simple neural architectures trained exclusively on analytical solutions.  
As a result, both training and inference can be carried out on standard hardware (e.g., a laptop), without requiring large-scale computational resources.
Further details of each component are provided in the following sections.

\section{Neural operator indicator function}
\label{section_NN_ind}

We now introduce a neural operator architecture based on Radial Basis Functions (RBFs), combining the original DeepONet \cite{lu2019} with classical RBF theory \cite{buhmann2003}. 
The proposed RBF-DeepONet is designed to map the full far-field matrix to a spatial indicator function representing the scatterer. 

\subsection{RBF-DeepONet}
\label{RBF_DeepONet}
\paragraph{DeepONets for inverse scattering}
We adopt the framework of neural operators, which generalize standard neural networks by learning mappings between function spaces.  
This choice is natural in our setting, as the inverse problem amounts to mapping the far-field operator
\begin{align*}
\mathcal{F} : L^2(\Sph^1) \to L^2(\Sph^1)
\quad \text{(see \cref{ff_operator} for details)}
\end{align*}
to an indicator function $I : \Omega \to \mathbb{R}$.
In practice, the data available in the LSM is represented by the far-field matrix $F$, which consists of pointwise evaluations of $\mathcal{F}$.  
This naturally motivates the use of the \emph{DeepONet} architecture \cite{lu2019}, which is designed to operate directly on pointwise samples of the input operator.
In our setting, the DeepONet takes the form
\begin{align}
I_{\theta}[F](z) := u_{\theta}(F)^{\top} v_{\theta}(\bs{z}),
\label{DeepONet}
\end{align}
where $F \in \mathbb{C}^{m_0 \times n_0}$ denotes the far-field matrix, $z \in \Omega$ a sampling point, $u_{\theta} : \mathbb{C}^{m_0 \times n_0} \to \mathbb{R}^{p}$ the branch network, and $v_{\theta} : \mathbb{R}^{2} \to \mathbb{R}^{p}$ the trunk network.  
Here, $p$ denotes the dimension of the output of both networks.
Note that the network can only process far-field matrices of size $m_0 \times n_0$. To overcome this limitation, any input far-field matrix with a different resolution $m \times n$ is first interpolated to the prescribed size $m_0 \times n_0$ using Fourier interpolation \cite{montanelli2015b,montanelli2017phd}.

\paragraph{RBF-DeepONets}
In the original DeepONet formulation, both the branch and trunk networks are modeled as trainable multilayer perceptrons (MLPs), resulting in a large number of trainable parameters.  
In our setting, however, the target outputs are positive indicator functions, which form a restricted class of functions.  
We therefore fix the trunk representation by prescribing a suitable set of basis functions, so that the trunk no longer depends on any trainable parameters.
The resulting architecture is illustrated in \cref{fig:DeepONet}.

The choice of basis functions is thus a critical modeling decision.  
We adopt RBFs, which have been extensively studied in the literature \cite{buhmann2003,fornberg2015,powell1987}.  
Introducing a set of center locations $\{z_i\}_{i=1}^{p}$ and a radial function $\phi_{\epsilon}(r)$, where $\epsilon>0$ is a width parameter, we define the basis functions by
\[
v^{i}(z) := \phi_{\epsilon}\!\left(\lvert z-z_{i} \rvert\right),
\qquad i=1,\ldots,p.
\]
In all experiments, we use Gaussian RBFs of the form $\phi_{\epsilon}(r)=\exp(-\epsilon r^{2})$, which are standard in many RBF-based approximation schemes.
It is also possible to enforce positivity of the DeepONet output by applying a positive function $\sigma_{+}$ to the output of the branch net.  
We refer to the resulting architecture as a \emph{RBF-DeepONet}.

\begin{figure}[t] 
\centering
\begin{tikzpicture}[
  node distance=1.1cm and 1.2cm,
  every node/.style={font=\small},
  netnode/.style={draw, rounded corners, minimum width=2.0cm, minimum height=0.7cm, align=center},
  input/.style={circle, draw, minimum size=0.6cm},
  output/.style={draw, minimum width=3.0cm, minimum height=1.0cm, align=center},
  arrow/.style={-Latex, thick}
]

\node[input] (F) {$F$};
\node[netnode, right=0.8cm of F] (branch) {Branch \\ $u_{\theta}(F)$};
\node[netnode, right=0.8cm of branch] (coeffs) {$u_{\theta}^1(F),\dots,u_{\theta}^p(F)$};

\node[input, below=1.5cm of F] (z) {$\bs{z}$};
\node[netnode, right=0.8cm of z] (trunk) {Trunk \\ $v(z)$};
\node[netnode, right=0.8cm of trunk] (basis) {$v^1(z),\dots,v^p(z)$};

\node[output, right=2.0cm of coeffs, yshift=-1.05cm] (out) 
{$I_{\theta}[F](z) = \displaystyle \sum_{i=1}^{p} u_{\theta}^i(F)\,v^i(z)$};

\draw[arrow] (F) -- (branch);
\draw[arrow] (z) -- (trunk);
\draw[arrow] (branch) -- (coeffs);
\draw[arrow] (trunk) -- (basis);

\coordinate (midpoint) at ($(out.west) + (-0.4cm,0)$); 

\draw[arrow] (coeffs.east) -| (midpoint) |- (out.west);
\draw[arrow] (basis.east) -| (midpoint) |- (out.west);

\node[above=0.2cm of branch] {\textbf{Trainable branch net}};
\node[above=0.2cm of trunk] {\textbf{Fixed RBF trunk net}};

\end{tikzpicture}
\caption{Schematic architecture of our RBF-DeepONet. The trainable branch net generates coefficients in the fixed trunk net basis.}
\label{fig:DeepONet}
\end{figure}
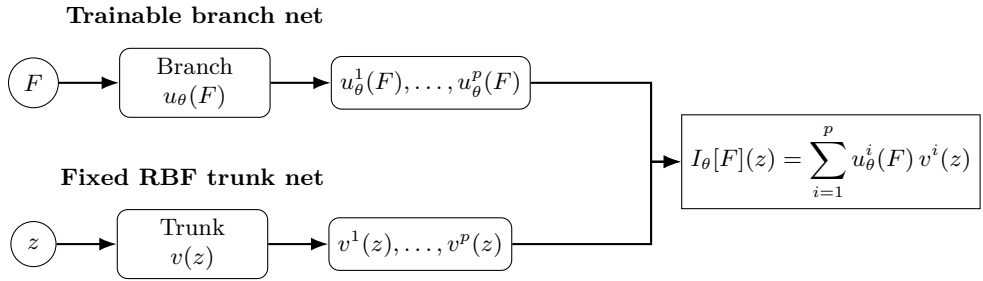

\subsection{Training}

We introduce a resolution parameter $h$, analogous to the mesh size in finite element methods. 
We found that $h = \lambda/2$ is a good compromise between accuracy and computational cost. 
Placing the basis functions uniformly in the square domain $\Omega = [-\lambda L,\, \lambda L]^2$ yields
\[
n_h
= \left\lfloor \frac{2\lambda L}{h} \right\rfloor + 1 \;\; \text{(basis functions per dimension)},
\quad
p_h = n_h^2 \;\; \text{(total number of basis functions)}.
\]

\paragraph{Training data}
The training data consist exclusively of far-field matrices $F_i$ corresponding to disks, which enables fast data generation since analytical far-field expressions are available (see \cref{appendix_a}). 
The disks are uniformly distributed in $\Omega$. 
Specifically, along each side of the square domain, we place $4n_h$ equispaced positions $c_i$ in both the $x$- and $y$-directions, yielding a total of $16p_h$ training samples. 
For each position $c_i$, the radius $r_i$ of the disk is drawn uniformly at random in the interval $[\lambda/2,\,1.5\lambda]$.
The target indicator function is defined by
\[
I[F_i](z) =
\begin{cases}
1, & \text{if } \lvert z - c_i \rvert \le r_i, \\
0, & \text{otherwise}.
\end{cases}
\]

\paragraph{Loss}
Let $\{F_i\}_{i=1}^{16p_h}$ denote the $m \times n$ training matrices and $\{I[F_i]\}_{i=1}^{16p_h}$ the corresponding label functions. 
We define the loss function as
\begin{align*}
\mathcal{L}(\theta)
= \frac{1}{16p_h^2}
\sum_{i=1}^{16p_h}
\sum_{j=1}^{p_h}
\bigl|
I_{\theta}[F_i](z_j) - I[F_i](z_j)
\bigr|^2 .
\end{align*}
The quadrature points used to measure the discrepancy between the network output $I_{\theta}[F_i]$ and the reference function $I[F_i]$ are chosen as the centers of the basis functions $\{z_j\}_{j=1}^{p_h}$.
All parameters are listed in \cref{training_info_DeepONet}. 
We found this setup to be a good compromise between computational cost and performance.

\begin{table}[t!]
\centering
\caption{Training setup retained for our RBF-DeepONet, achieving a good compromise between computational speed and performance.}
\label{training_info_DeepONet}
\begin{tabular}{lc}
\toprule
Resolution parameter & $h=\lambda/2$\\
Number of basis functions & $p_h=(\lfloor 2 \lambda L/h \rfloor + 1)^2$\\
Training obstacles type & disks \\
Number of training samples & $16p_h$ \\
Training  positions $c_i$ & uniform grid on $\Omega$ \\
Training  radii $r_i$ & random in $[\lambda/2,1.5\lambda]$ \\
Label training functions & binary indicator \\
\bottomrule
\end{tabular}
\end{table}

\subsection{Neural Tangent Kernel analysis}

We choose Gaussian radial basis functions of the form 
$\phi_{\epsilon}(r) =\exp(-\epsilon r^2)$. 
The physical problem imposes a characteristic wavelength $\lambda$, and we fix the resolution parameter to $h = \lambda/2$. 
The remaining design parameter is therefore the shape parameter $\epsilon$ of the Gaussian basis functions.
To relate $\epsilon$ to the resolution scale $h$, we introduce an overlap parameter $s \in (0,1)$ defined by
\[
\phi_{\epsilon}(h) = s,
\]
which measures the value of a basis function at the nearest neighboring center. 
This yields the explicit relation
\[
\epsilon = -\frac{\log s}{h^2}.
\]
The parameter $s$ thus directly controls the overlap between adjacent basis functions: 
smaller values of $s$ lead to more localized Gaussians (larger $\epsilon$), while larger values of $s$ produce smoother and more overlapping basis functions. 
This behavior is illustrated in 1D in \cref{fig:rbf_1D_comparison}.

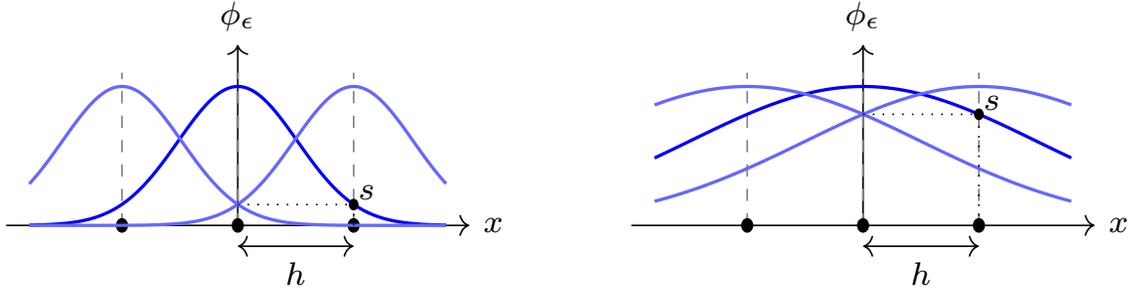
\begin{figure}[t]
\centering
\begin{minipage}{0.45\textwidth}
\centering
\resizebox{\textwidth}{!}{%
\begin{tikzpicture}[yscale=1.2, font=\footnotesize] 

\def\h{1.0}   
\def\s{0.15}
\pgfmathsetmacro{\eps}{-ln(\s)/(\h*\h)}

\draw[->] (-2,0) -- (2,0) node[right] {$x$};
\draw[->] (0,0) -- (0,1.3) node[above] {$\phi_\epsilon$};

\foreach \c in {-1,0,1} {
    \draw[dashed, gray] (\c*\h,0) -- (\c*\h,1.1);
    \fill (\c*\h,0) circle (1.5pt);
}

\draw[thick, blue, domain=-1.8:1.8, samples=200] plot (\x,{exp(-\eps*(\x)^2)});
\draw[thick, blue!60, domain=-1.8:1.8, samples=200] plot (\x,{exp(-\eps*(\x-\h)^2)});
\draw[thick, blue!60, domain=-1.8:1.8, samples=200] plot (\x,{exp(-\eps*(\x+\h)^2)});

\draw[<->] (0,-0.15) -- (\h,-0.15);
\node at (\h/2,-0.35) {$h$};
\draw[dotted] (\h,0) -- (\h,\s);
\draw[dotted] (0,\s) -- (\h,\s);
\fill (\h,\s) circle (1.2pt);
\node[anchor=south west, inner sep=1pt] at (\h,\s) {$s$};

\end{tikzpicture}%
}
\end{minipage}%
\hfill 
\begin{minipage}{0.45\textwidth}
\centering
\resizebox{\textwidth}{!}{%
\begin{tikzpicture}[yscale=1.2, font=\footnotesize] 

\def\h{1.0}    
\def\s{0.8}
\pgfmathsetmacro{\eps}{-ln(\s)/(\h*\h)}

\draw[->] (-2,0) -- (2,0) node[right] {$x$};
\draw[->] (0,0) -- (0,1.3) node[above] {$\phi_\epsilon$};

\foreach \c in {-1,0,1} {
    \draw[dashed, gray] (\c*\h,0) -- (\c*\h,1.1);
    \fill (\c*\h,0) circle (1.5pt);
}

\draw[thick, blue, domain=-1.8:1.8, samples=200] plot (\x,{exp(-\eps*(\x)^2)});
\draw[thick, blue!60, domain=-1.8:1.8, samples=200] plot (\x,{exp(-\eps*(\x-\h)^2)});
\draw[thick, blue!60, domain=-1.8:1.8, samples=200] plot (\x,{exp(-\eps*(\x+\h)^2)});

\draw[<->] (0,-0.15) -- (\h,-0.15);
\node at (\h/2,-0.35) {$h$};
\draw[dotted] (\h,0) -- (\h,\s);
\draw[dotted] (0,\s) -- (\h,\s);
\fill (\h,\s) circle (1.2pt);
\node[anchor=south west, inner sep=1pt] at (\h,\s) {$s$};

\end{tikzpicture}%
}
\end{minipage}

\caption{One-dimensional Gaussian radial basis functions with identical center spacing $h$ and two different overlap parameters: $s=0.15$ (left) and $s=0.8$ (right). Larger values of $s$ yields broader overlap between basis functions, whereas smaller values produce more localized functions.}
\label{fig:rbf_1D_comparison}
\end{figure}

To guide us in our choice of overlap parameter $s$, we utilize of the Neural Tangent Kernel (NTK) theory. First introduced in \cite{jacot2018}, the NTK provides a framework to describe the training dynamics of neural networks by considering gradient flow, where the training parameters evolve continuously as $\theta=\theta(t)$. For a network $f_\theta : \mathbb{R}^d \to \mathbb{R}^{d'}$ with parameters 
$\theta \in \mathbb{R}^P$, trained on a dataset $\{(x_i, y_i)\}_{i=1}^N$ 
with $y_i \in \mathbb{R}^{d'}$, the empirical vector training error can be written as
\[
e(t) = 
\begin{pmatrix}
f_\theta(x_1) - y_1\\
\vdots\\
f_\theta(x_N) - y_N
\end{pmatrix} \in \mathbb{R}^{N d'}.
\]
Under gradient flow, its evolution satisfies
\[
\frac{d}{dt} e(t) = - K(t)\, e(t),
\]
where $K(t) \in \mathbb{R}^{N d' \times N d'}$ is the Neural Tangent Kernel, 
defined from the gradients of the network evaluated at training points (see \cref{appendix:ntk_classic}). 
In general, $K(t)$ depends on the parameters $\theta(t)$ and is thus time-dependent, 
but in the infinite-width limit, it converges to a constant deterministic kernel 
$K_\infty$, reducing the training dynamics to an autonomous linear  system. 
This perspective provides qualitative insight into convergence and generalization 
properties of neural networks.

In the case of a RBF-DeepONet with $p$ basis functions and $N$ training samples, 
$K $
can be written
\[
K = \bigl(I_{N} \otimes P_{\epsilon}^\top\bigr)\, G \,\bigl(I_{N} \otimes P_{\epsilon}\bigr) \in \mathbb{R}^{(N p) \times  (Np)},
\]
where $P_\epsilon \in \mathbb{R}^{p\times p}$ is defined  by
\[
(P_\epsilon)_{ij}
=
\phi_\epsilon\!\left(|z_j-z_i|\right),
\qquad i,j=1,\dots,p,
\] and $G \in \mathbb{R}^{(Np)\times(Np)} $ is the NTK matrix of the trainable branch net (see \cref{appendix:ntk_DeepONet} for details). 
The factorization of $K$ leads to bounds regarding its spectrum.
\begin{theorem}
\label{th_NTK}
The following bounds for the spectrum of $K$ holds:

$$
\lambda_{\min}(G)\, \sigma_{\min}(P_{\epsilon})^{2} 
\leq \lambda_{\min}(K) 
\leq \lambda_{\max}(K) 
\leq \lambda_{\max}(G)\, \sigma_{\max}(P_{\epsilon})^{2},
$$
where $\sigma_{\min}(P_{\epsilon})$ and $\sigma_{\max}(P_{\epsilon})$ denote the
minimum and maximum singular values of $P_{\epsilon}$. 
\end{theorem}
The proof can be found in \cref{appendix:ntk_DeepONet}.
We seek to design a network architecture that minimizes the conditioning of the kernel matrix $K$, thus allowing efficient training. According to the bound above, this requires reducing the condition numbers of both $P_{\epsilon}$ and $G$. In particular, this motivates careful selection of the branch-network architecture, both in terms of width and depth, as well as the choice of the overlap parameter $s$ for the radial basis functions. We illustrate this with NTK numerical experiments in a setting where $\lambda=1$, $h = 1/2$, $p = 121$, and $N = 1936$, without any positivity-enforcing function, so that the time-independent NTK regime holds. To reduce the computational cost, all calculations are performed on a single batch, which results in matrices of size $p \times p$. 

\begin{figure}[!t] \centering \def\scl{0.37}
\includegraphics[scale=\scl]{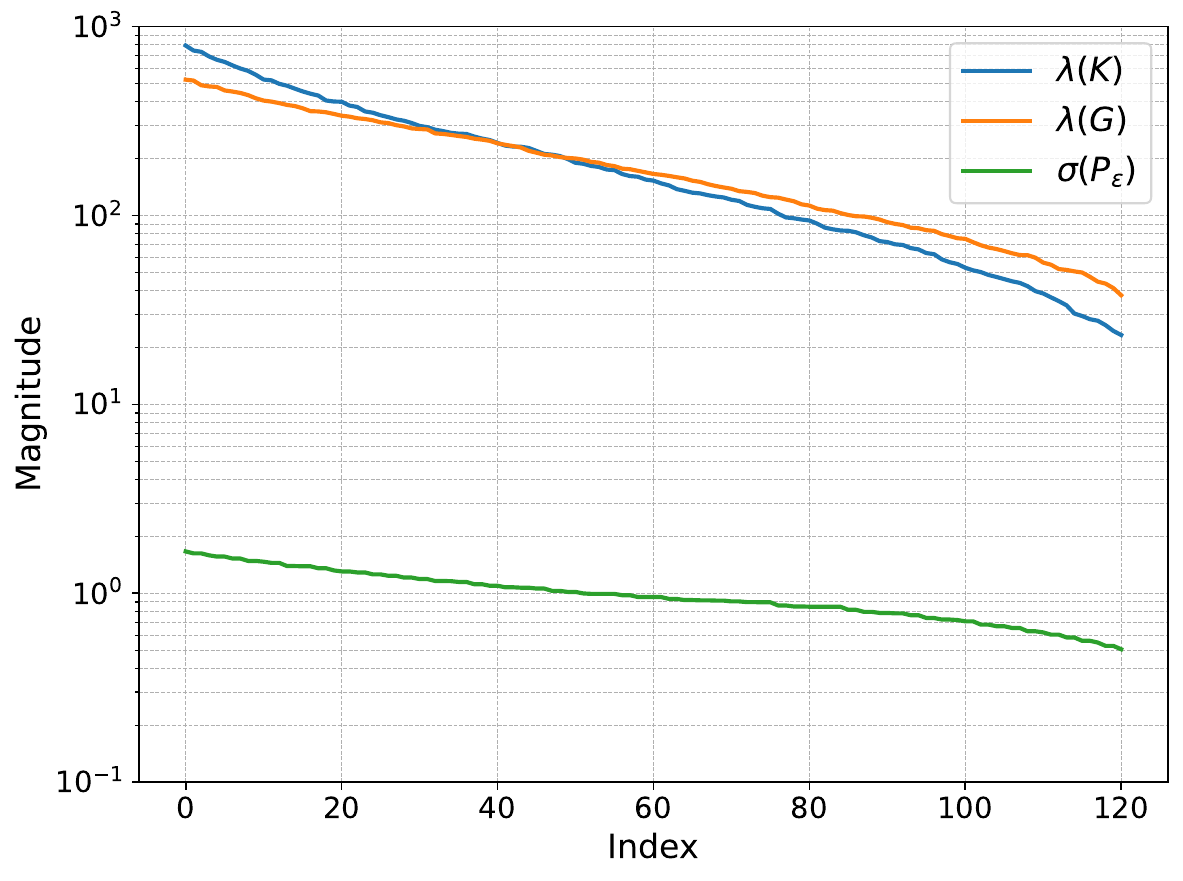} 
\includegraphics[scale=\scl]{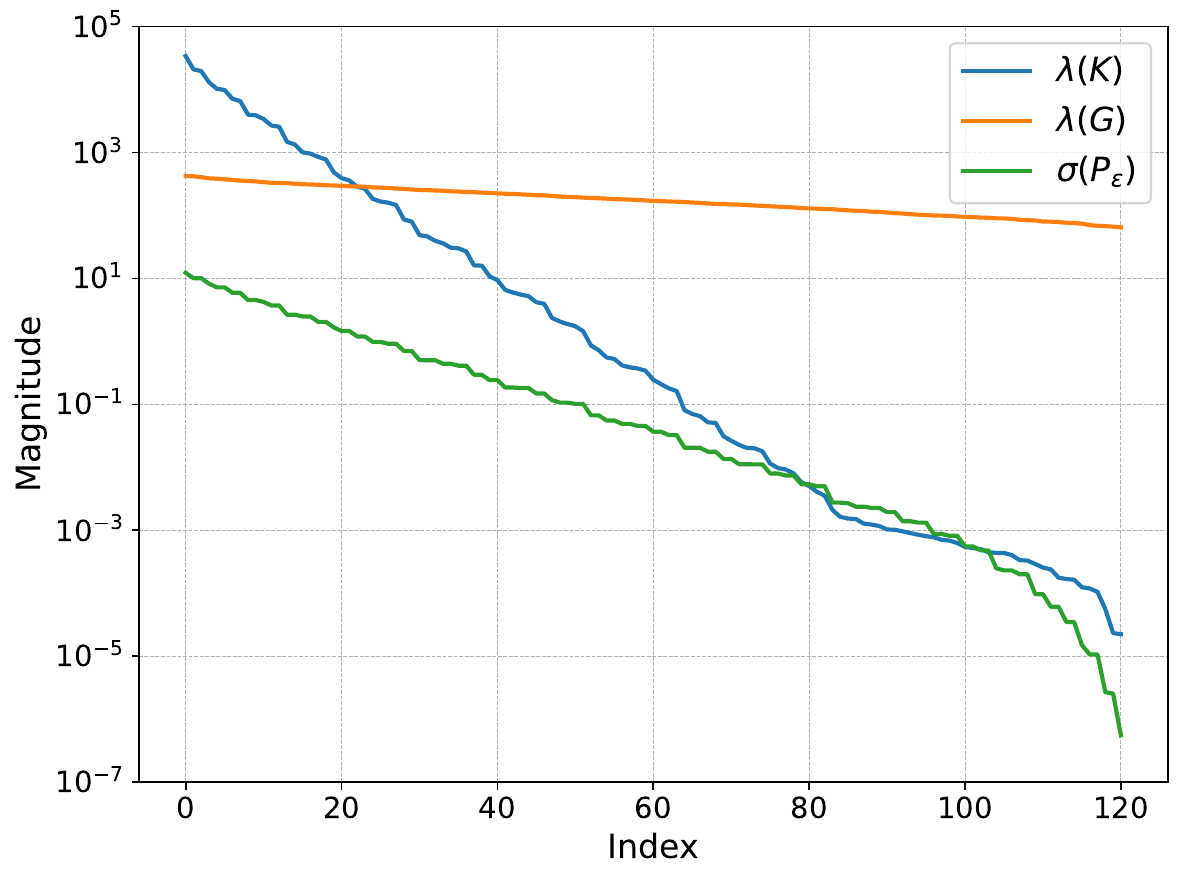} 
\includegraphics[scale=\scl]{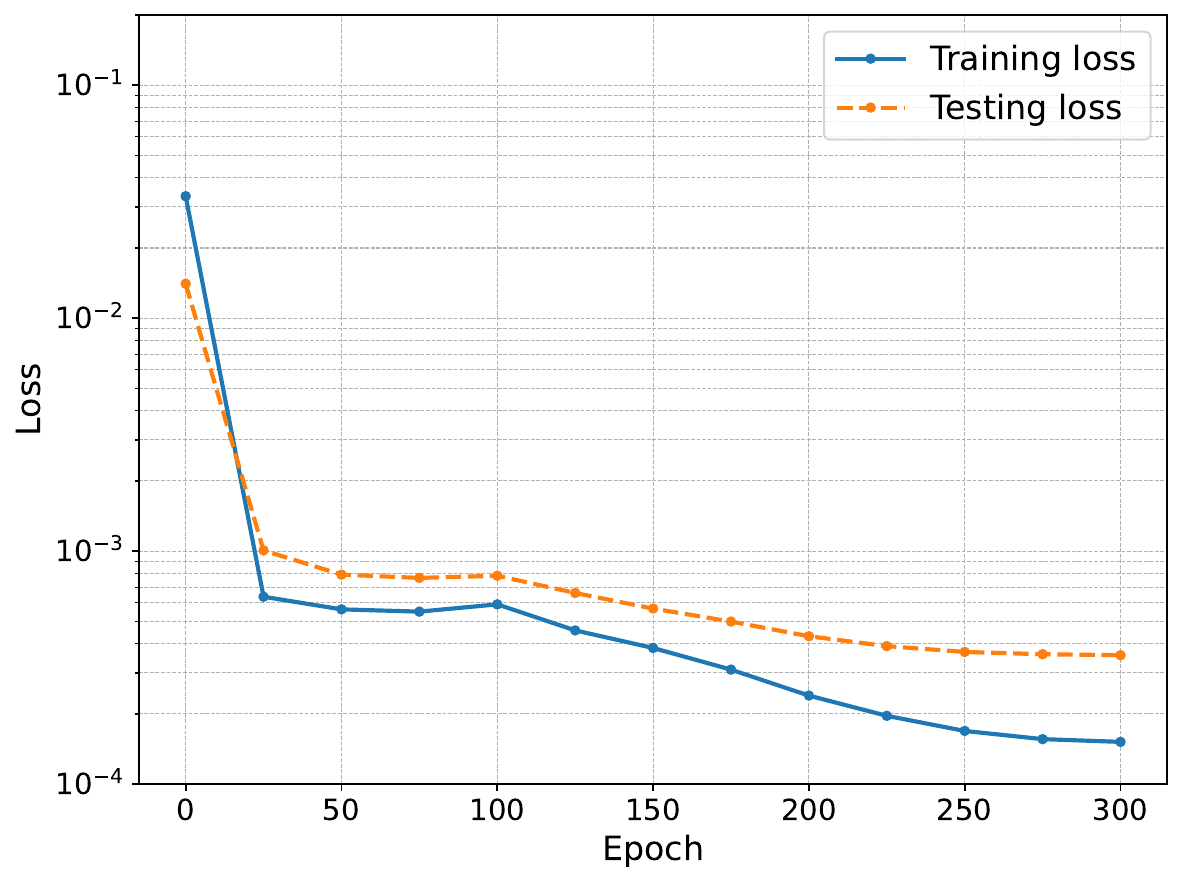} 
\includegraphics[scale=\scl]{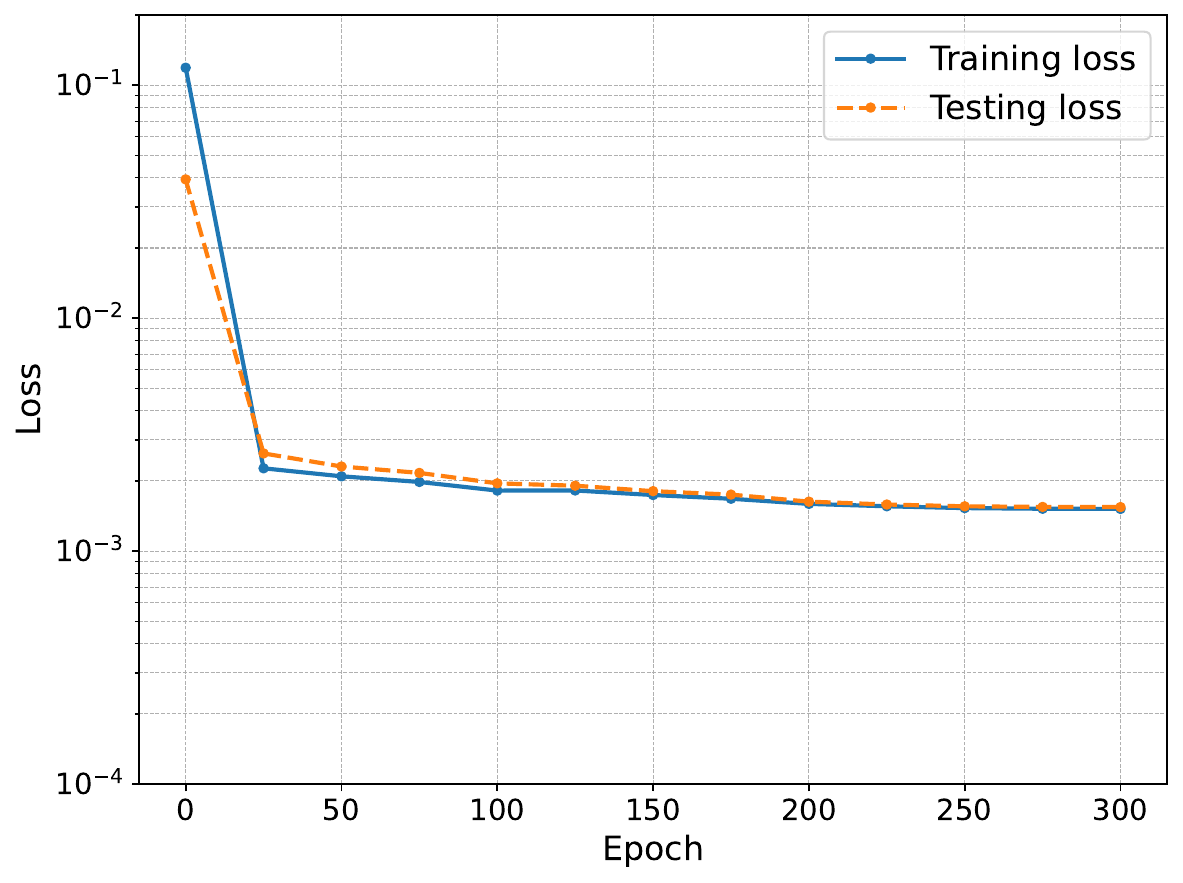} 
\caption{
NTK spectra ($K$, $G$, $P_{\epsilon}$) and training histories for $s=0.15$ (left) and $s=0.8$ (right).
For $s=0.15$, $P_{\epsilon}$ is well conditioned, yielding a well-behaved NTK spectrum and improved training.
For $s=0.8$, $P_{\epsilon}$ is highly ill conditioned, which deteriorates the spectrum of $K$ and hampers training.
}
\label{fig:NTK_s} 
\end{figure}

\paragraph{Influence of the overlap parameter $s$}
We first study the influence of the overlap parameter $s$. 
As shown in \cref{fig:NTK_s}, excessively small values of $s$ result in severe ill-conditioning of the kernel matrix $K$.
While the matrix $P_{\epsilon}$ is positive definite for all $\epsilon$ (see \cite{fornberg2015}),  it can easily be shown that $\kappa(P_{\epsilon})$, the condition number of $P_{\epsilon}$,
satisfies
$$
\kappa(P_{\epsilon}) \underset{\epsilon \to 0}{\longrightarrow} \infty
\quad \text{and} \quad
\kappa(P_{\epsilon}) \underset{\epsilon \to \infty}{\longrightarrow} 1.
$$
Thus, relying solely on NTK analysis, the theoretically optimal choice would be $\epsilon \to \infty$, i.e., $s \to 0$.
However, while very large values of $\epsilon$ favor training, it comes at the cost of poor generalization, 
since the corresponding Gaussians have little overlap and fail to cover the full domain $\Omega$. A practical stability condition can be obtained by requiring the sum of two neighboring Gaussians not to have a local minimum at the midpoint, which is equivalent to
$ s \geq e^{-2} \approx 0.135. $ In practice, we choose $s=0.15$.

\paragraph{About the positivity-enforcing function $ \sigma_+$}
We have observed that enforcing positivity by applying a positive function $\sigma_{+}$ to the output of the branch net leads to more stable results. Positivity is also a desirable property, as the network output is subsequently used as a regularization function.
Although NTK theory provides insight into how various parameters influence training dynamics, enforcing positivity in the branch network through the application of a positive function moves the training dynamics beyond the standard NTK regime. Nevertheless, we observe that parameter choices suggested by NTK analysis still yield near-optimal performance under positivity constraints. After testing several positive functions, we find that the square function produces the most stable results.

Based on extensive numerical experiments, we summarize in \cref{tab:optimal_DeepONet} the RBF-DeepONet configuration that achieves favorable training behavior while keeping the number of trainable parameters reasonable.

\begin{table}[t!]
\centering
\caption{RBF-DeepONet architecture and training hyperparameters selected for optimal performance.}
\label{tab:optimal_DeepONet}
\begin{tabular}{lc}
\toprule

\multicolumn{2}{l}{\textbf{Architecture}} \\
\midrule
Branch network depth & 1 hidden layer \\
Branch network width & $3p$ \\
Branch activation function & tanh \\
Positivity enforcement $\sigma_+$ & square function \\
Trunk basis functions & Gaussian radial basis functions \\
Overlap parameter & $s = 0.15$ \\

\midrule
\multicolumn{2}{l}{\textbf{Training}} \\
\midrule
Optimizer & Adam with weight decay $5 \times 10^{-5}$ \\
Learning rate & $10^{-3}$ with cosine annealing to $10^{-5}$ \\
Training iterations  & 300 epochs \\

\bottomrule
\end{tabular}
\end{table}

\section{Noise level prediction}
\label{section_noise}

This section addresses the estimation of the noise level $\delta$. 
An estimate of $\delta$ is essential, as it directly determines the amount of regularization: lower noise levels require less regularization. 
In practice, however, $\delta$ is unknown and must be inferred from the data.
We propose a neural-network-based estimator that takes as input the singular values of the far-field matrix. 
Using the full matrix as input proved inefficient, and we were unable to obtain satisfactory results with this approach. 
An additional advantage of relying on singular values is their invariance with respect to translations of the obstacle: due to the relation in \cref{invariance_relation}, the singular values of the far-field operator do not depend on the obstacle's position.

\subsection{Decay of singular values}
Our approach is based on the decay of the singular values of the far field matrix $F_{\delta}$. Considering a sound-soft circular obstacle with radius $R$, equation \cref{asympotics_dirichlet} provides the expected decay rate for the singular values of the far-field operator:
\begin{align*}
  \sigma_{p} \underset{p \rightarrow+\infty}{\sim}  \sqrt{\frac{8 \pi^3}{k}} \frac{1}{p!(p-1)!}\left( \frac{kR}{2} \right) ^{2p}.  
\end{align*}
Note that the singular values converge to $0$. This is a consequence of the compactness of the far-field operator (see definition \cref{ff_operator}), which is typical in the context of inverse scattering problems. At the discrete level, this property leads to ill-conditioned matrices. In the present case, the situation is really severe, as the eigenvalues decay super-exponentially.  However, the presence of noise perturbs the singular values (see eigenvalue perturbation theory~\cite{crandall1973,simon1982}), thereby altering the observed decay. In \cref{fig:eig}, we display the singular values of $F_{\delta}$ for different values of the noise parameter $\eta$. It corresponds to a sound-soft circular obstacle with radius $0.5$.

\begin{figure}[!t]
\centering
\def\scl{0.4}
\includegraphics[scale=\scl]{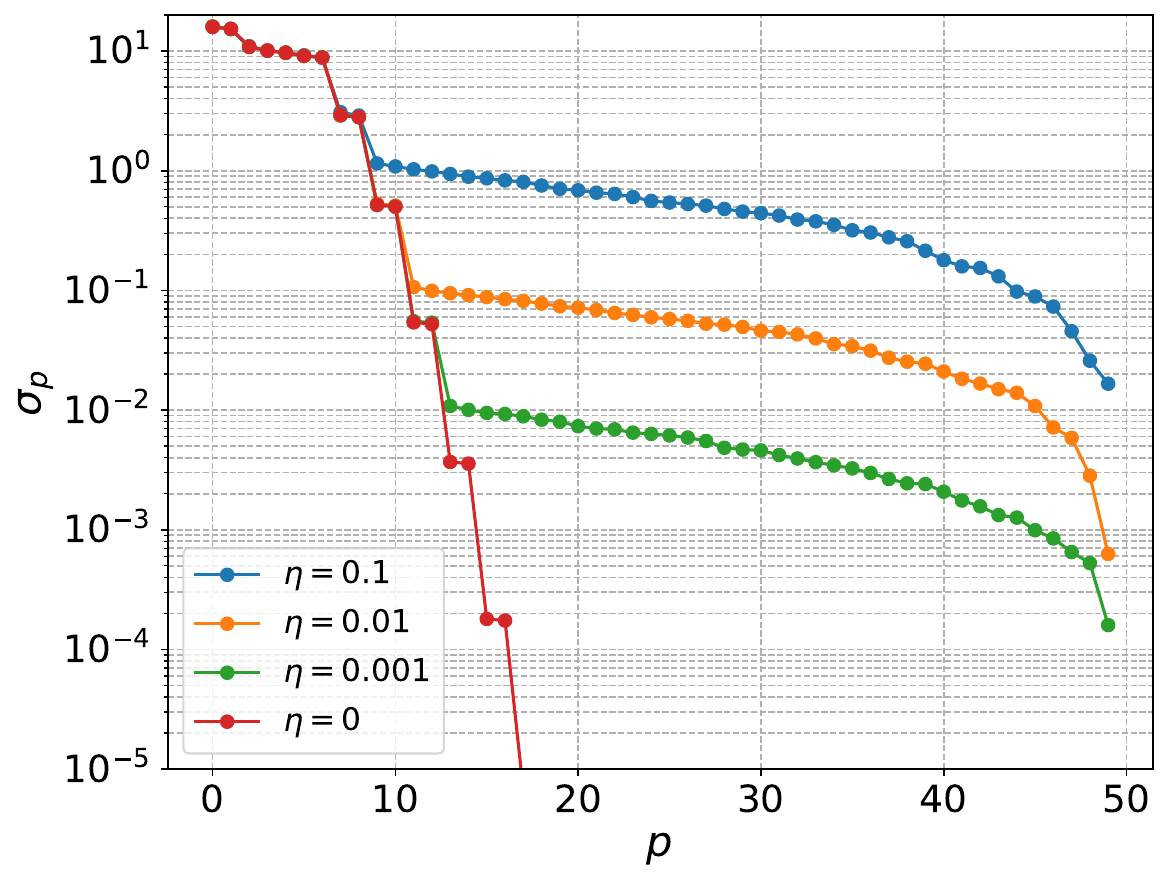} 
\caption{Eigenvalues of \(50 \times 50\) noisy far-field matrices for a circular obstacle of radius \(0.5\), computed with noise levels \(\eta = 0.1\) (blue), \(\eta = 0.01\) (orange), \(\eta = 0.001\) (green), and \(\eta = 0\) (red), where $\eta$ is defined in \cref{noising}.
}
\label{fig:eig}
\end{figure}
We observe that the singular values, especially the smaller ones, provide information about the noise level. In particular, the presence of noise prevents the singular values from converging toward zero, as they appear to converge to a noise-dependent plateau. This behavior may be related to the Marchenko--Pastur distribution, which describes the asymptotic singular value distribution of random matrices, although we have not further investigated this point.

\subsection{Neural network approach}
\label{network_noise}

We investigate the use of neural networks to predict the noise level 
\[
\delta = \|F - F_\delta\|_2,
\] 
from the singular values of the far-field matrix. 
Similarly to what was described in \cref{RBF_DeepONet}, given a noisy far-field matrix $F_\delta \in \mathbb{C}^{m \times n}$, 
we first interpolate it onto a prescribed size $m_0 \times n_0$ using Fourier interpolation, 
yielding $\tilde{F}_\delta \in \mathbb{C}^{m_0 \times n_0}$. 
The singular values of $\tilde{F}_\delta$ are then used to predict the noise level using a neural network.

\paragraph{Training data} 
For training, we generate exclusively $m_0 \times n_0$ far-field matrices. 
Training noise amplitude $\eta$ are sampled from a log-uniform distribution over $[5\times 10^{-3}, 3\times 10^{-1}]$. 
For each $\eta$, a random obstacle radius $r \in [\lambda/2, 1.5\lambda]$ is selected, and the corresponding $n_0 \times m_0$ noisy far-field matrix for a circle of radius $r$ is constructed.
As network features, we use the logarithms of its singular values: 
\[
\left(\log \sigma_1, \log \sigma_2, \dots, \log \sigma_{n_0}\right)^\top \in \mathbb{R}^{n}.
\]
The logarithm is used to improve scaling properties.
The network labels are defined as
\[
\log\left(\frac{\delta}{\sqrt{m_0} + \sqrt{n_0}}\right),
\]
where $\delta$ denotes the noise level. Once again, the logarithm ensures appropriate scaling, while the factor $1/(\sqrt{m_0} + \sqrt{n_0})$ is motivated by random matrix theory: for a matrix with independent, zero-mean entries, the spectral norm typically scales like $\sqrt{m} + \sqrt{n}$ (see \cite[Thm.~4.4.5]{vershynin2018}).

\paragraph{Evaluation} 
Given a noisy far-field matrix $F_\delta$ of shape $m_0 \times n_0$, we first reshape it into $\tilde{F}_\delta \in \mathbb{C}^{m \times n}$ using Fourier interpolation. 
The network input is the logarithm of the singular values of $\tilde{F}_\delta$. 
The predicted noise level $\delta_\theta$ is obtained by taking the exponential of the network output and multiplying by $\sqrt{m_0} + \sqrt{n_0}$ to account for the original matrix size.

\paragraph{Choice of the network}
For the noise-level estimator, we employ a standard feed-forward neural network (MLP). 
Despite their simplicity, MLPs remain among the most widely used and versatile models in machine learning. 
The data generation and training procedure is summarized in \cref{training_info_noise}.

\begin{table}[t!]
\centering
\caption{Training configuration and noise prediction for the network.}
\label{training_info_noise}
\begin{tabular}{lc}
\toprule
\multicolumn{2}{l}{\textbf{Training data}} \\
\midrule
Number of samples & 400 noisy far-field matrices \\
Obstacle type & sound-soft circle centered at the origin, radius $r \in [\lambda/2,\, 1.5\lambda]$ \\
Matrix size & $m_0 \times n_0$ (fixed) \\
Input features & $\bigl(\log \sigma_1, \log \sigma_2, \dots, \log \sigma_{n_0} \bigr)$, singular values of $F_\delta$ \\
Output labels & $\log \left( \dfrac{\delta}{\sqrt{m_0} + \sqrt{n_0}} \right)$, with $\delta = \|F - F_\delta\|_2$ \\
Noising process &  Gaussian noise: $F_\delta = F \cdot (1 + \eta X)$, $\eta \sim \mathrm{LogUniform}[5\times10^{-3},\, 3\times 10^{-1}]$ \\

\midrule
\multicolumn{2}{l}{\textbf{Testing data}} \\
\midrule
Number of samples & 50 per test obstacle \\
Matrix size & $m \times n$ (possibly different from $m_0 \times n_0$) \\
Preprocessing & Fourier interpolation to $m_0 \times n_0$ before feature extraction \\
Network input & $\bigl(\log \sigma_1, \dots, \log \sigma_{n_0} \bigr)$ of interpolated matrix \\
Noise prediction & $\delta_\theta = (\sqrt{m} + \sqrt{n}) \exp(\text{network output})$  \\

\midrule
\multicolumn{2}{l}{\textbf{Network configuration}} \\
\midrule
Network type & MLP \\
Activation function & ReLU \\
Architecture & $(n_0,\,100,\,1)$ \\
Optimizer & Adam with weight decay $10^{-4}$ \\
Learning rate & $5 \times 10^{-3}$ \\
Training iterations & 300 epochs \\
\bottomrule
\end{tabular}
\end{table}

\paragraph{Numerical results}
The performance of the network is evaluated on far-field matrices of varying sizes and corresponding to obstacles of different shapes. 
Representative predictions are displayed in \cref{fig_noises}, while quantitative results, reported in terms of the mean relative error $|\delta - \delta_\theta|/|\delta|$, are given in \cref{tab_noise_errors}. 
In all computations, we fix $m_0 = n_0 = 30$, i.e., each far-field matrix is interpolated to a standardized $30 \times 30$ resolution.

Across all test cases, the mean relative error is typically below $20\%$. 
This accuracy is sufficient for our purposes, since the estimated noise level is only used to guide regularization and parameter selection; capturing the correct order of magnitude is enough. 
The network provides stable estimates across different obstacle geometries and matrix sizes, and similar performance is observed when Gaussian noise is replaced by uniform noise. We therefore expect the method to extend to other zero-mean noise distributions.

The approach nevertheless has limitations. 
First, generalization outside the training noise range is poor, as expected for a purely data-driven method, which highlights the importance of selecting an appropriate training interval. 
Second, accurate prediction requires sufficiently many singular values. For larger obstacles, the decay is slower and the noise-dominated regime appears at higher indices; insufficient resolution may therefore obscure the noise plateau. In particular, interpolating to a $30 \times 30$ matrix is inadequate for obstacles larger than $3\lambda$.

\begin{figure}[!t]
\centering
\def\scl{0.315}

\begin{tabular}{cc}
\textbf{Circle} & \textbf{Kite} \\[2mm]

\includegraphics[scale=\scl]{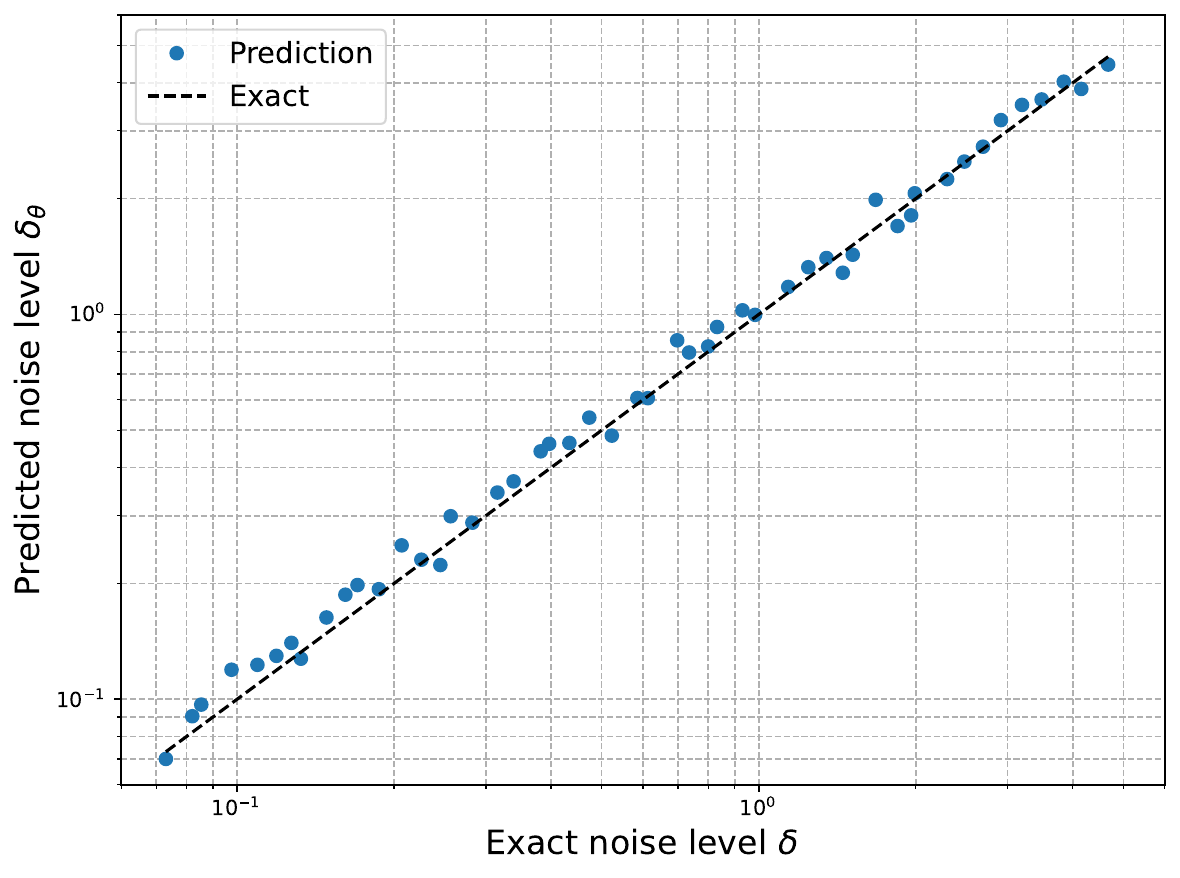} &
\includegraphics[scale=\scl]{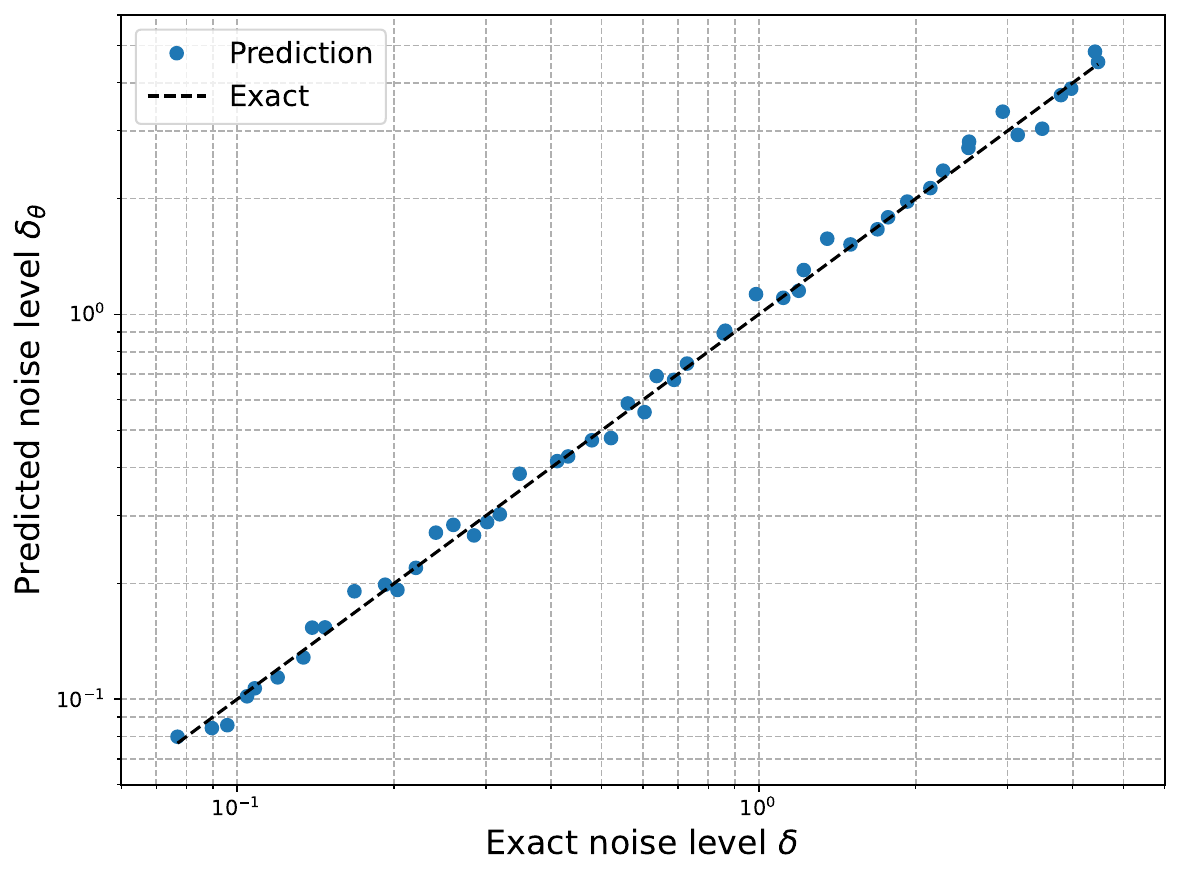} \\

\small $(60 \times 30)$ & \small $(60 \times 30)$ \\[3mm]

\includegraphics[scale=\scl]{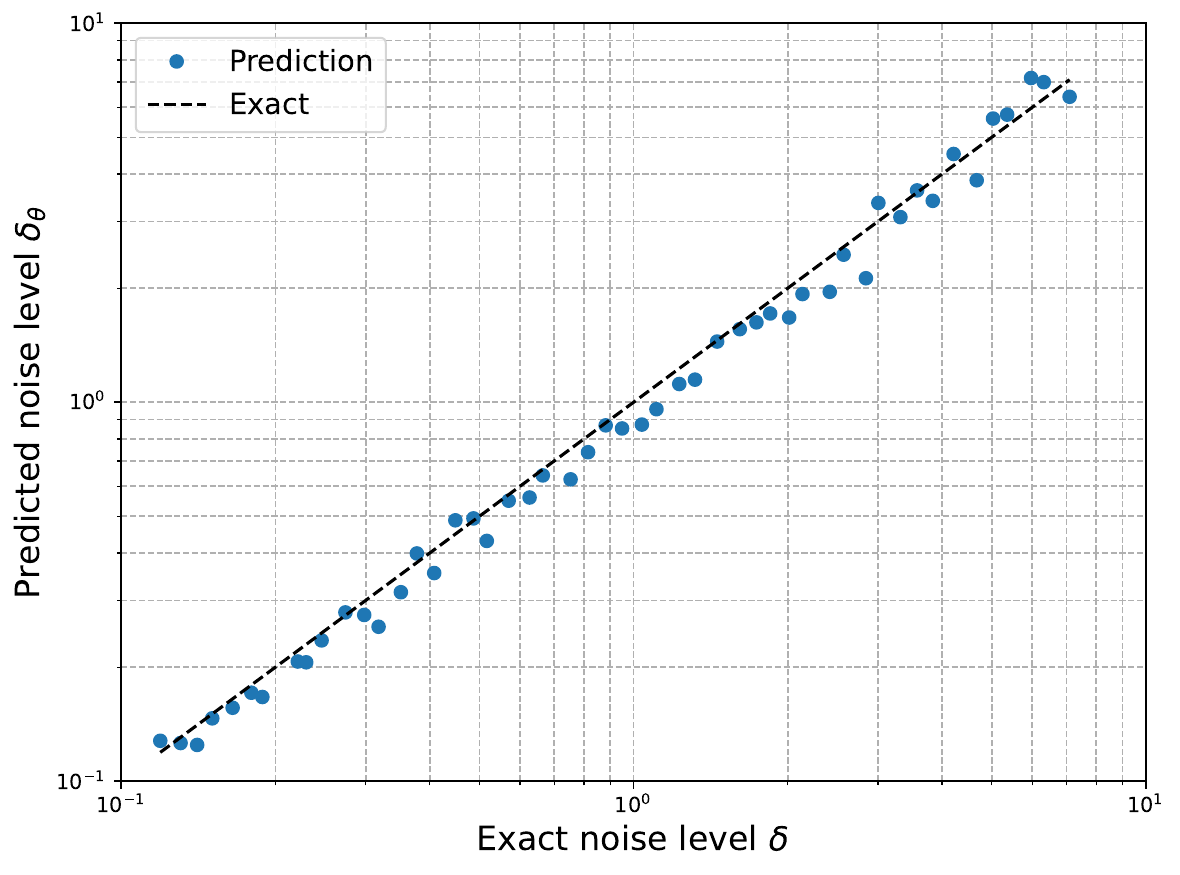} &
\includegraphics[scale=\scl]{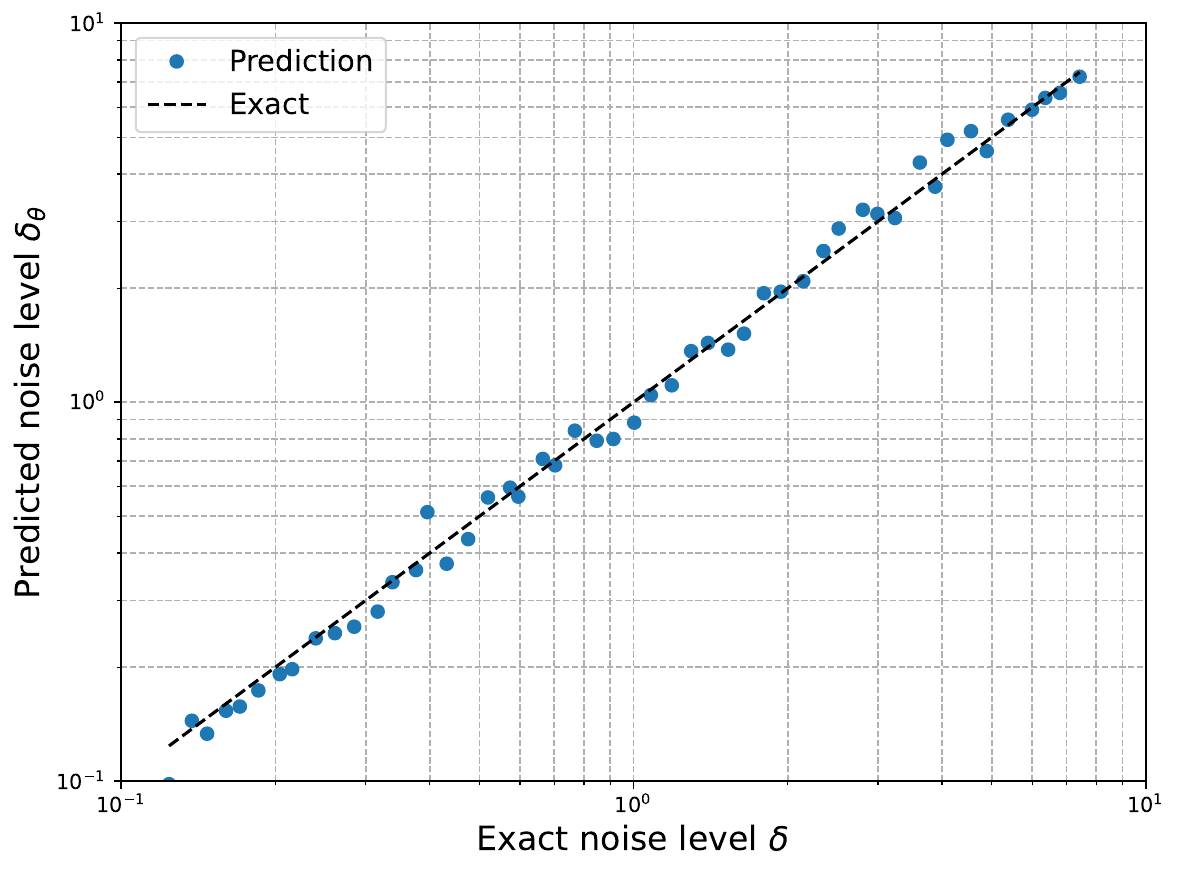} \\

\small $(100 \times 100)$ & \small $(100 \times 100)$
\end{tabular}

\caption{Exact versus predicted noise levels for different obstacle geometries and far-field matrix sizes.
Left column: circle wirth radius $r=1$.
Right column: kite.
Top row: initial far-field matrix of size $60 \times 30$.
Bottom row: initial far-field matrix of size $100 \times 100$. Predictions align well with exact noise level overall.}
\label{fig_noises}
\end{figure}

\begin{table}[t!]
\centering
\caption{Mean relative errors for different obstacle geometries and matrix sizes (typically $\approx 10\%$).}
\label{tab_noise_errors}
\begin{tabular}{lcc}
\toprule
\textbf{Matrix size} & \textbf{Obstacle} & \textbf{Mean relative error} \\
\midrule
\multirow{3}{*}{$100 \times 100$}
 & Circle ($r=0.60$) & $4.09 \times 10^{-2}$ \\
 & Circle ($r=1.30$) & $1.33 \times 10^{-1}$ \\
 & Kite ($r=0.80$)   & $1.21 \times 10^{-1}$ \\
\midrule
\multirow{3}{*}{$50 \times 50$}
 & Circle ($r=0.60$) & $6.99 \times 10^{-2}$ \\
 & Circle ($r=1.30$) & $1.34 \times 10^{-1}$ \\
 & Kite ($r=0.80$)   & $1.06 \times 10^{-1}$ \\
\bottomrule
\end{tabular}
\end{table}


\section{Regularizing the LSM}
\label{section_reg}

In this section, we seek to provide a practical strategy for selecting the regularization coefficients in the LSM, as an alternative to the computationally expensive Morozov's principle.

\subsection{Proposed approach}

A straightforward option is to use a single regularization parameter $\alpha$, independent of both the noise level $\delta$ and the sampling point $\bs{z}$. 
For instance, \cite{catapano2007} suggest choosing $\alpha = \|F\|_2 / 100$. 
Although this heuristic can yield satisfactory reconstructions across various obstacle types and noise levels, we observed that it may lead to suboptimal results in certain cases, particularly when the noise level is very low or when the obstacle is large. 
We therefore aim to design a more practical heuristic that explicitly accounts for both the noise level and the sampling point.

A natural guiding principle is that the regularization coefficient should decrease as the noise level decreases. Moreover, when inspecting the regularization maps obtained from the Morozov's principle (see \cref{fig_morozov}), we observe that the corresponding coefficients tend to be larger inside the obstacle. Motivated by this, we propose to exploit the indicator function predicted by the DeepONet to define a spatially varying regularization function. This indicator encodes spatial information and only requires appropriate scaling with respect to the noise level to serve as a suitable regularization function within the LSM. Thus, we propose choosing the regularization function as 
\[
\alpha_\theta[F_\delta](z) = \delta_\theta[F_\delta]\,
I_\theta[F_\delta](z),
\]
where  $F_{\delta}$ is the noisy far-field matrix, $z$ is the sampling point and $\delta_{\theta}$ is the predicted noise level, as described in \cref{section_noise}.
\subsection{Numerical results}
We present some numerical results. In particular, we compare the initial indication function $I_{\theta}$ obtained using the RBF-DeepONet and its associated LSM indicator functions $I$, using $\alpha_{\theta}$ as regularization function. For reference, we also display $I_{\mathrm{Mor}}$, the standard LSM indicator regularized with Morozov's principle.
\paragraph{Setup}
We consider initial matrices with shape $n_0\times n_0= 50 \times 50 $ interpolated into shape $n\times n= 30 \times 30 $.
We consider a full-aperture configuration, where sources and sensors are uniformly distributed around the entire obstacle. The experimental setup follows the description in \cref{LSM_setup}. We plot the initial neural network indicator $I_{\theta}(\bs{z})$, the associated LSM indicator $I(\bs{z})$ and the Morozov indicator $I_{\mathrm{Mor}}(\bs{z})$. 

\paragraph{Influence of the noise level}
We begin by analyzing the impact of noise on the reconstruction (see \cref{fig:full_noise_comparison}). The results indicate that the Morozov indicator is sensitive to the noise level, whereas the network-based indicators and the corresponding LSM indicator remain largely unaffected.
\begin{figure}[!t]
    \def\scl{0.32}
    $$
\begin{array}{ccc}
   I_{\theta}(\bs{z})  \text{: DeepONet } &   I(\bs{z}) \text{: LSM with DeepONet }     & I_{\mathrm{Mor}}(\bs{z}) \text{: LSM with Morozov}  \\
   
   \includegraphics[scale=\scl]{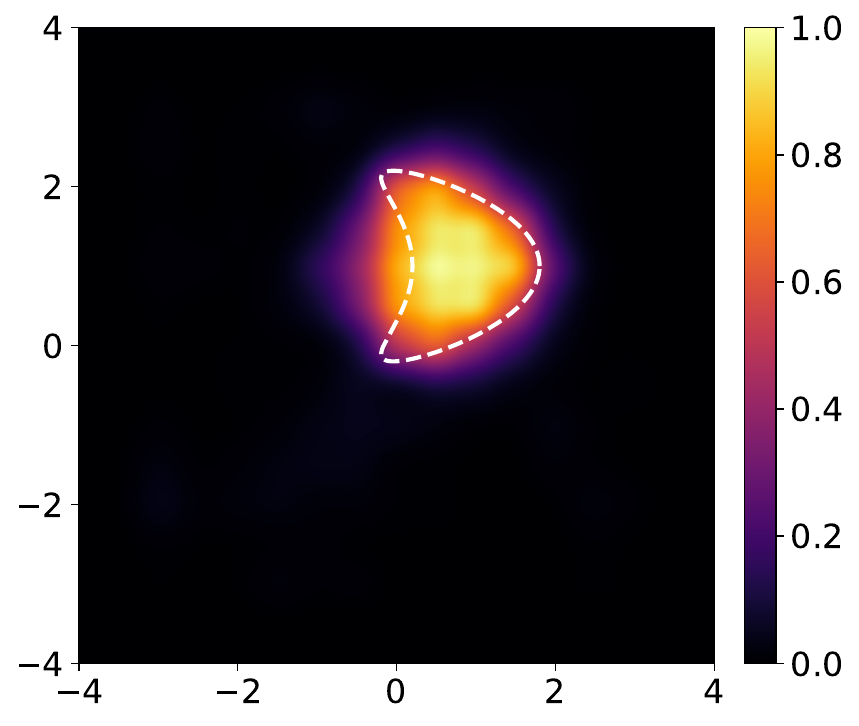}&
 \includegraphics[scale=\scl]{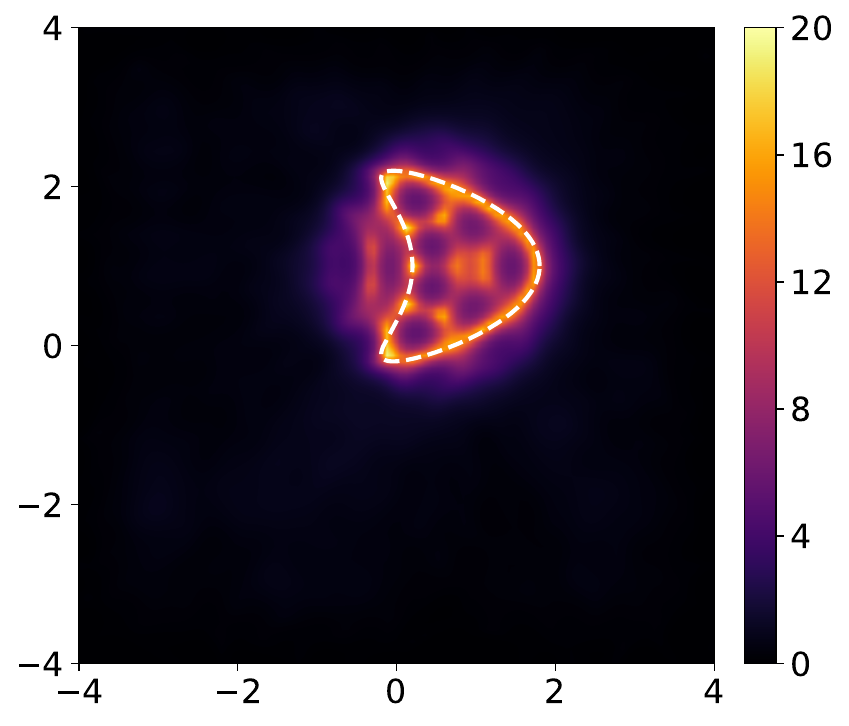} &
 \includegraphics[scale=\scl]{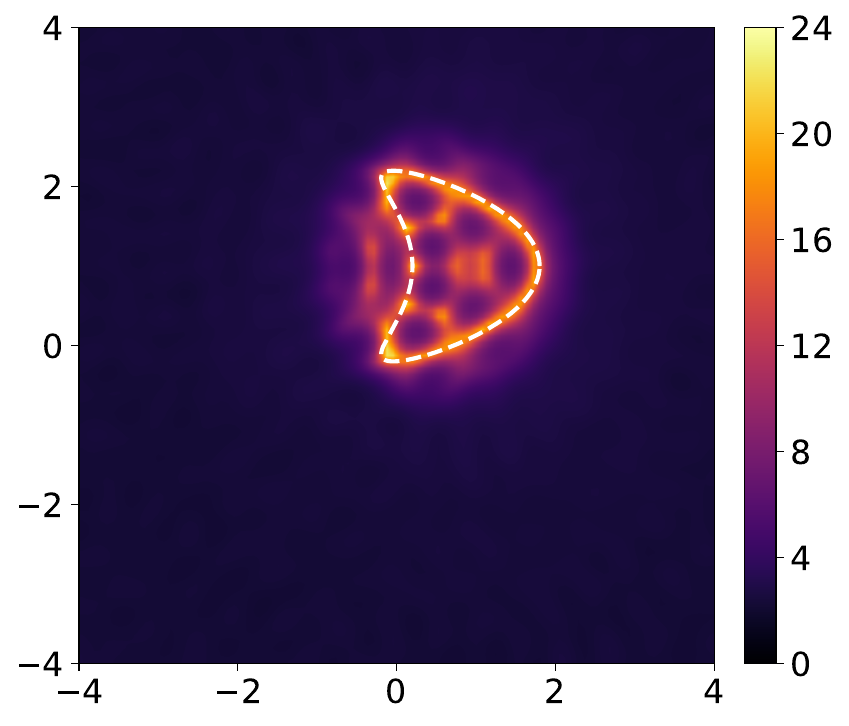} \\

  \includegraphics[scale=\scl]{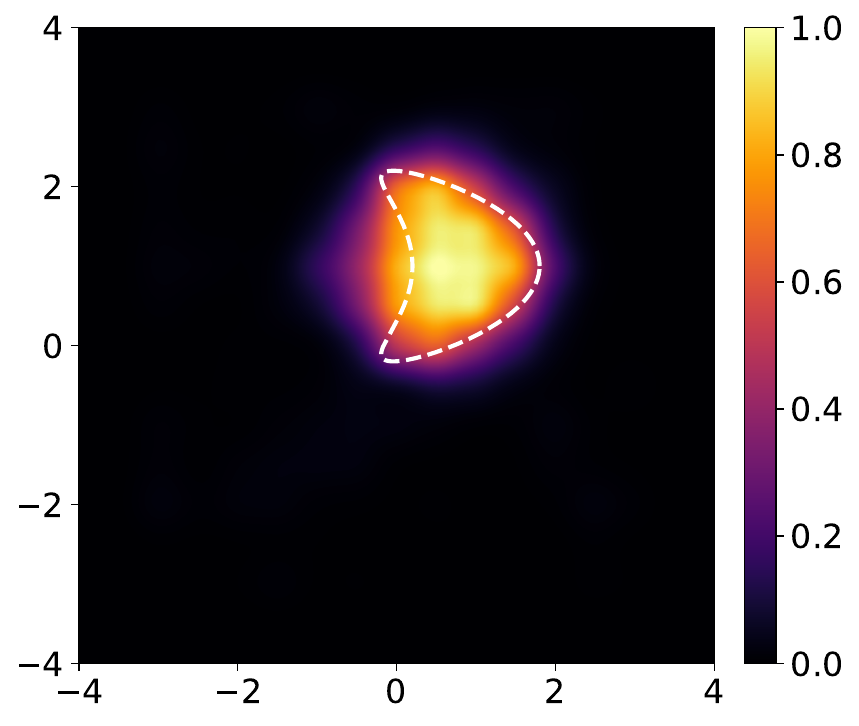}&
 \includegraphics[scale=\scl]{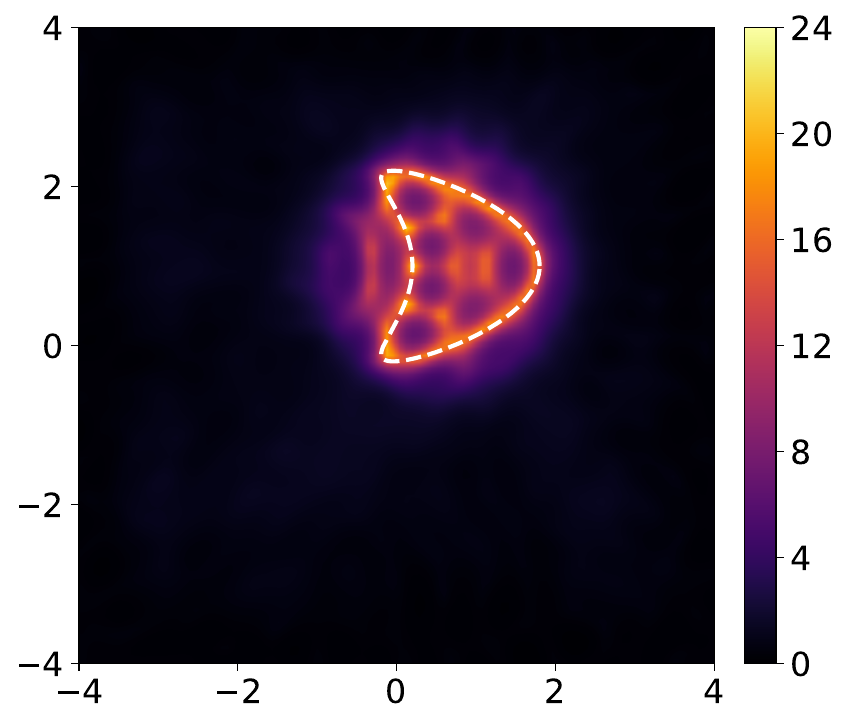} &
\includegraphics[scale=\scl]{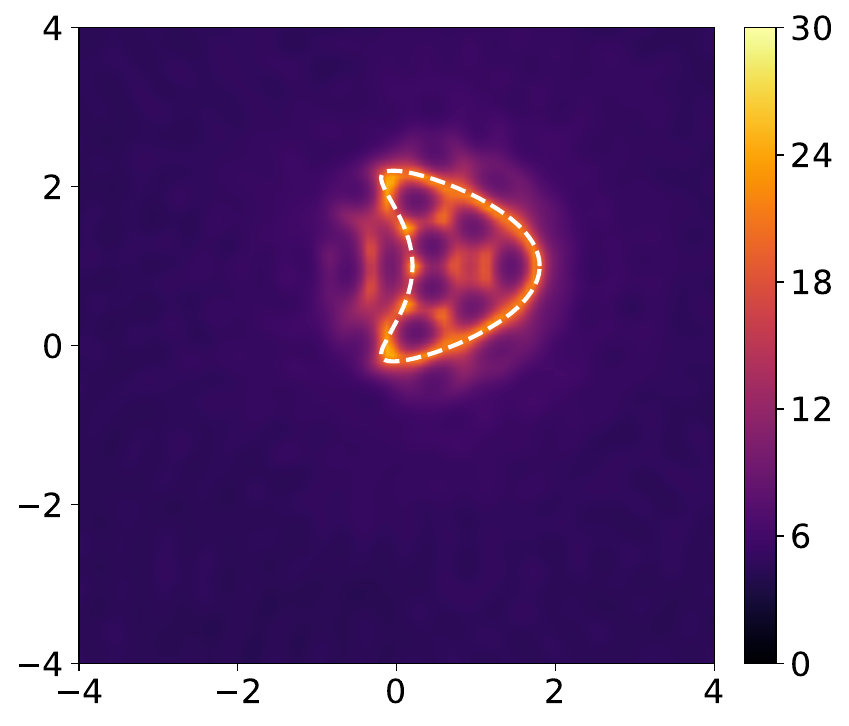} \\

  \includegraphics[scale=\scl]{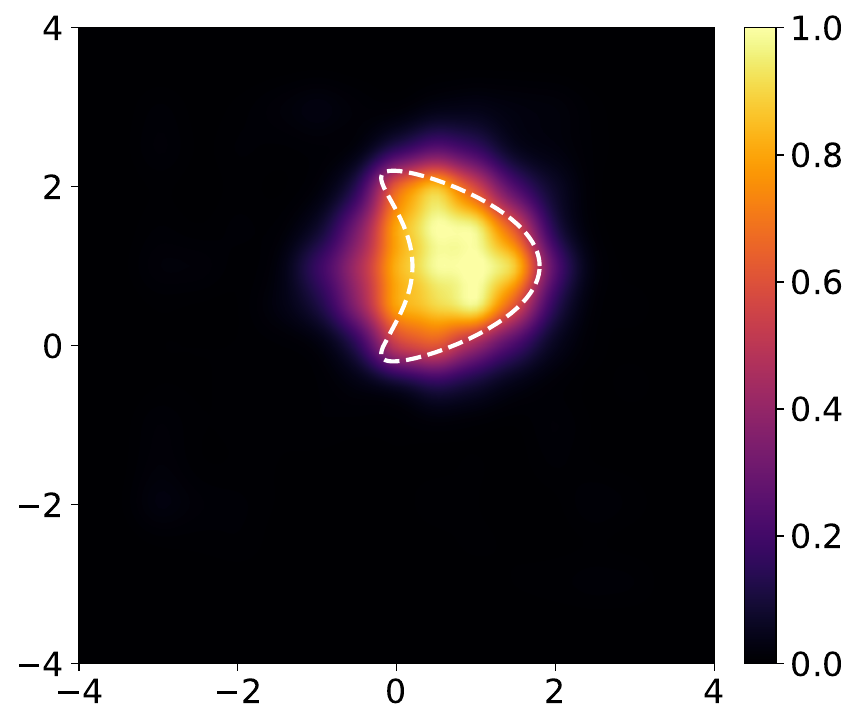}&
 \includegraphics[scale=\scl]{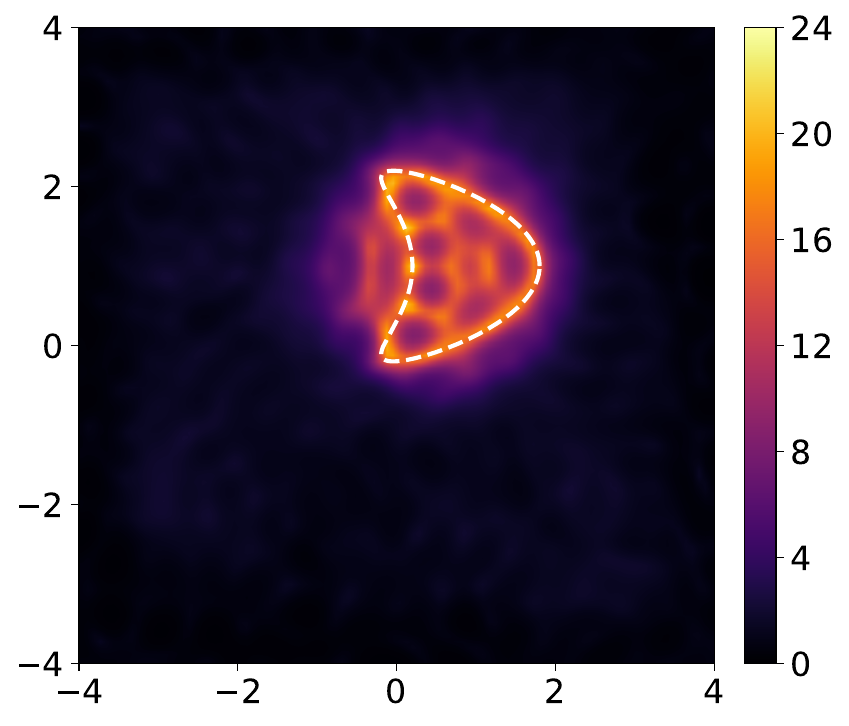} &
 \includegraphics[scale=\scl]{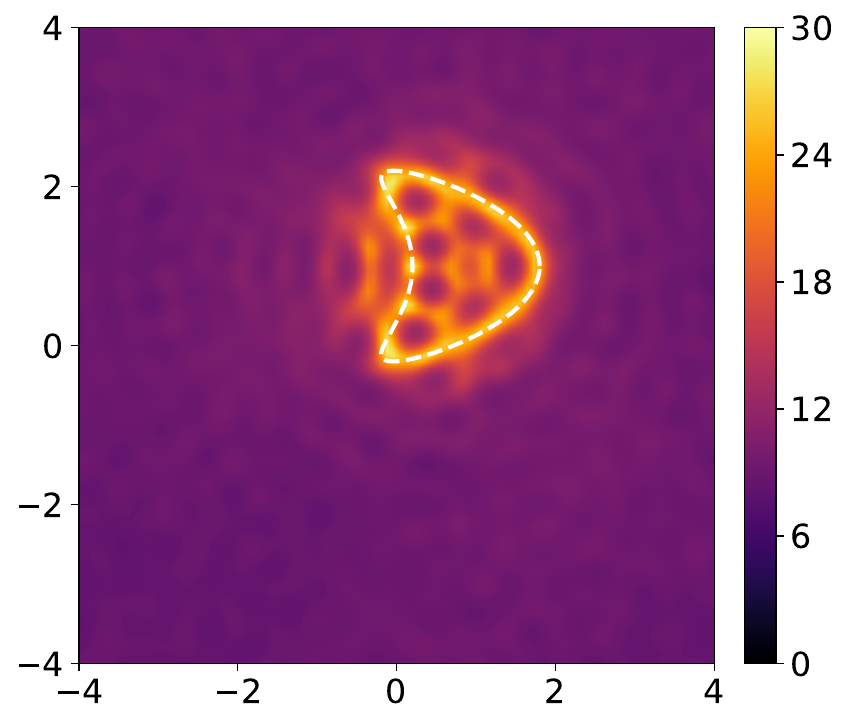} \\
\end{array}  $$
\caption{ Comparison of the DeepONet-based indicator (left column), the corresponding LSM indicator (middle column) and the Morozov LSM (right column) for a kite in full aperture. 
The noise parameter is set respectively to $\eta = 0.05$ (first row), $\eta = 0.1$ (second row) and $\eta = 0.2$ (fourth row). The DeepONet output seems minimally affected by the noise and the associated LSM indicator provides better contrast compared to Morozov LSM, especially for higher noise levels.
  }
\label{fig:full_noise_comparison}
\end{figure}

\paragraph{Single scatterers}
\Cref{fig:indicators_single} presents several reconstructions for different scatterer shapes, located at different positions. The DeepONet successfully retrieves the location and size of the defects. The corresponding LSM indicator enhances the reconstruction quality by refining the scatterer geometry and yielding higher contrast than the Morozov indicator.
\begin{figure}[!t]
    \def\scl{0.32}
    $$
\begin{array}{ccc}
   I_{\theta}(\bs{z})  \text{: DeepONet } &   I(\bs{z}) \text{: LSM with DeepONet }     & I_{\mathrm{Mor}}(\bs{z}) \text{: LSM with Morozov}  \\
   
   \includegraphics[scale=\scl]{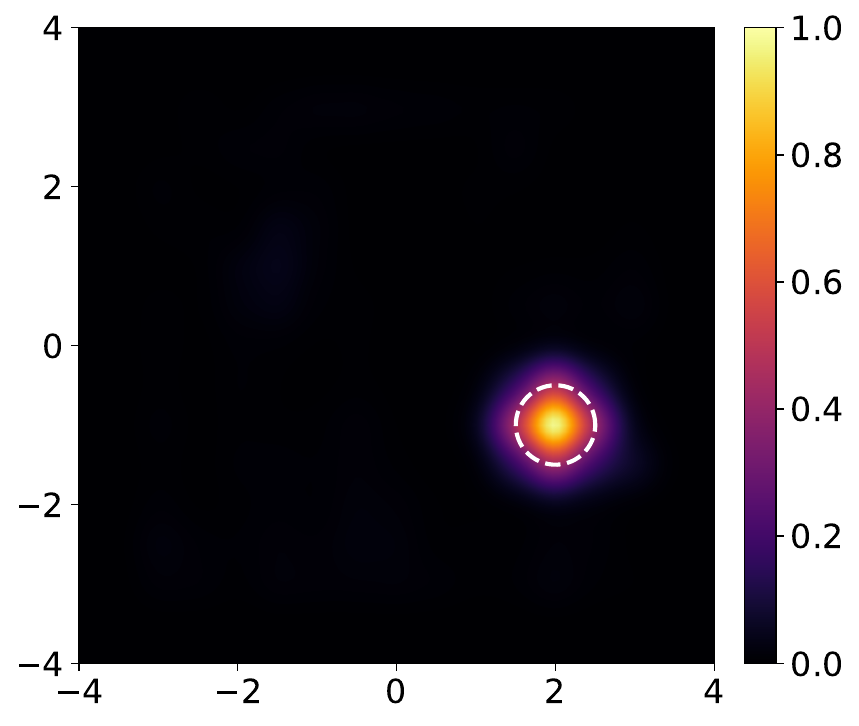}&
 \includegraphics[scale=\scl]{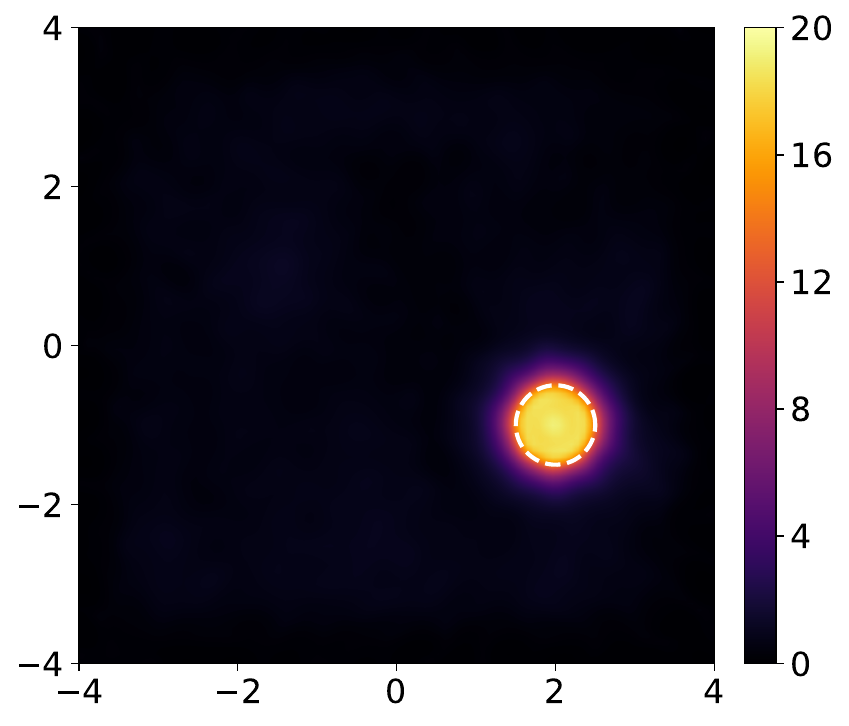}&
 \includegraphics[scale=\scl]{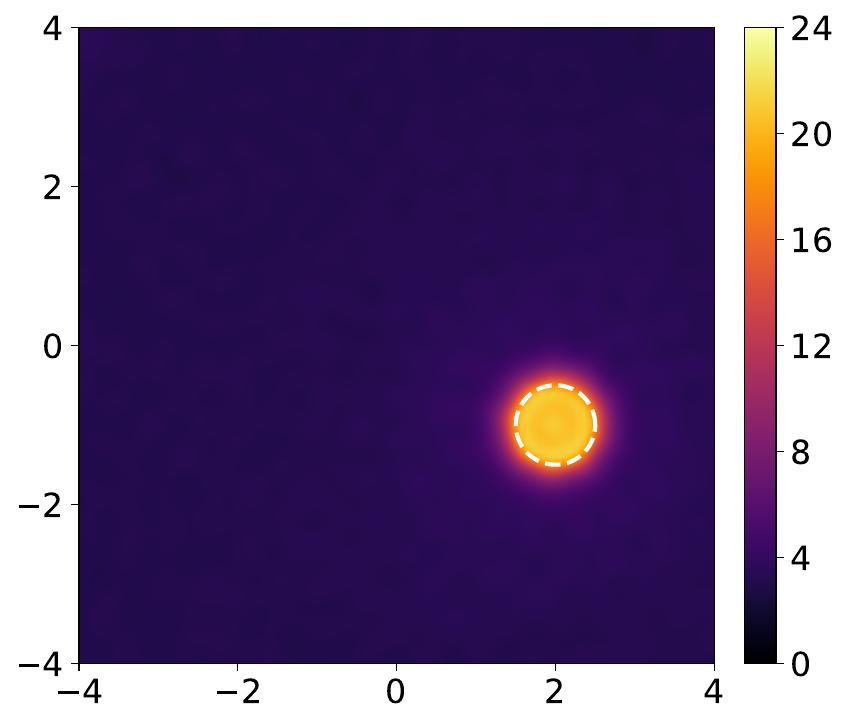} \\

   \includegraphics[scale=\scl]{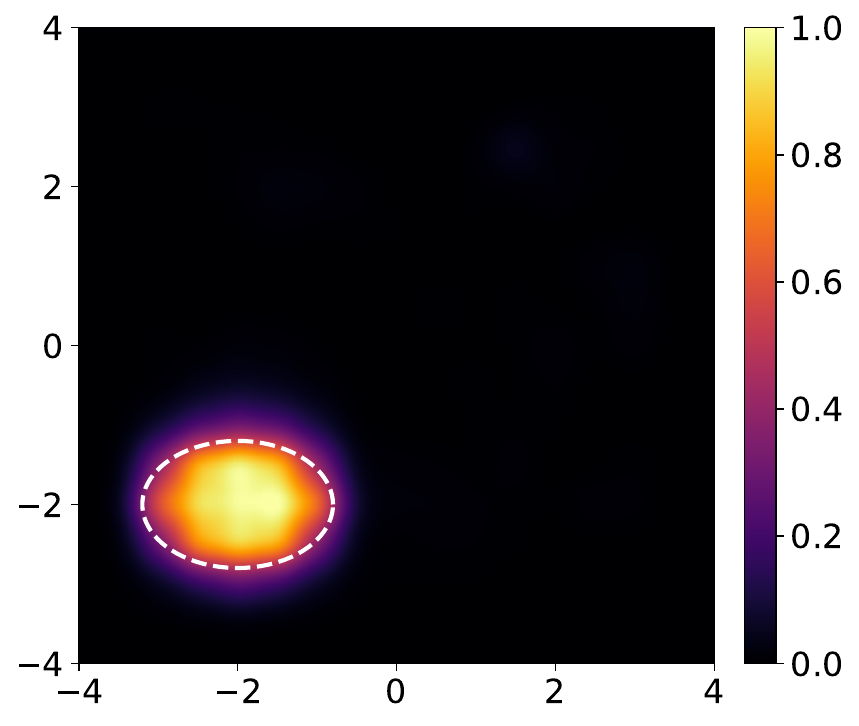}&
 \includegraphics[scale=\scl]{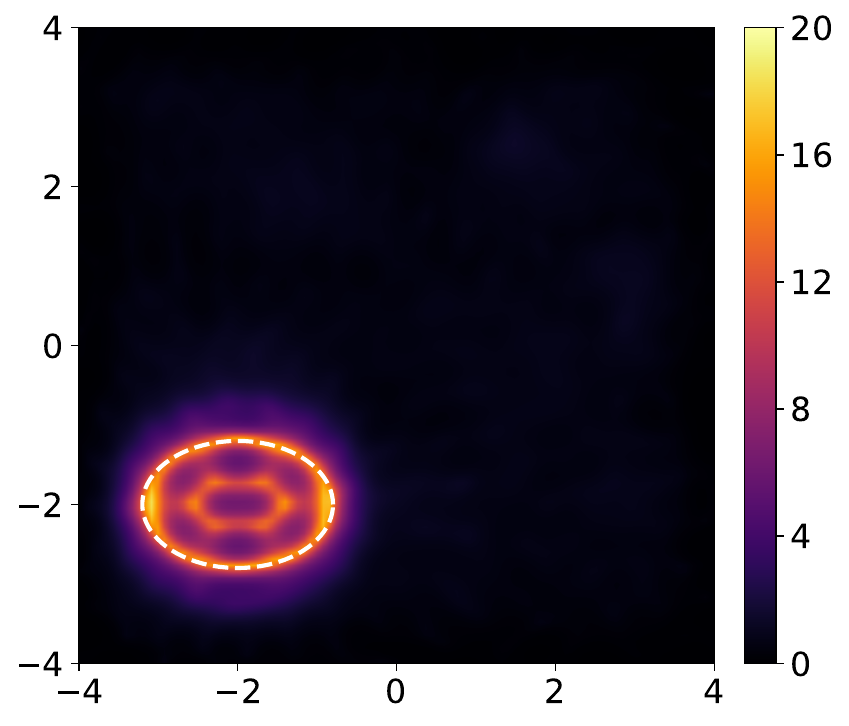}&
 \includegraphics[scale=\scl]{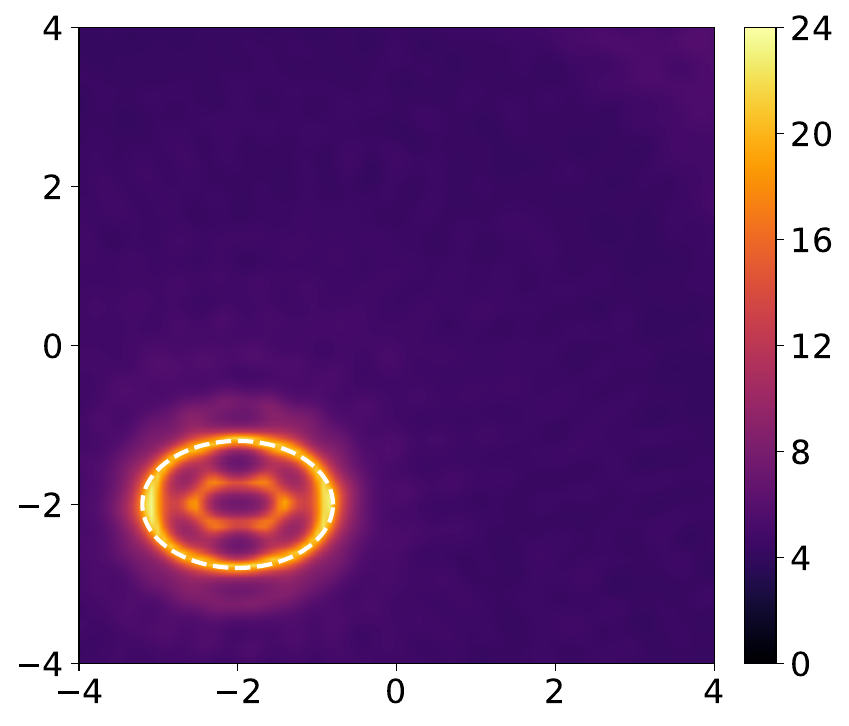} \\

   \includegraphics[scale=\scl]{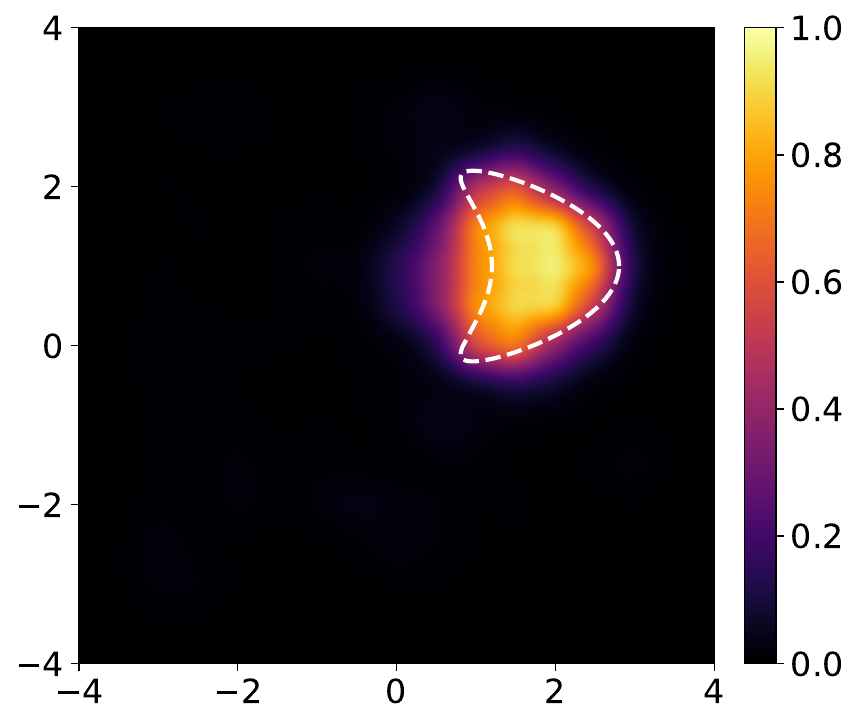}&
 \includegraphics[scale=\scl]{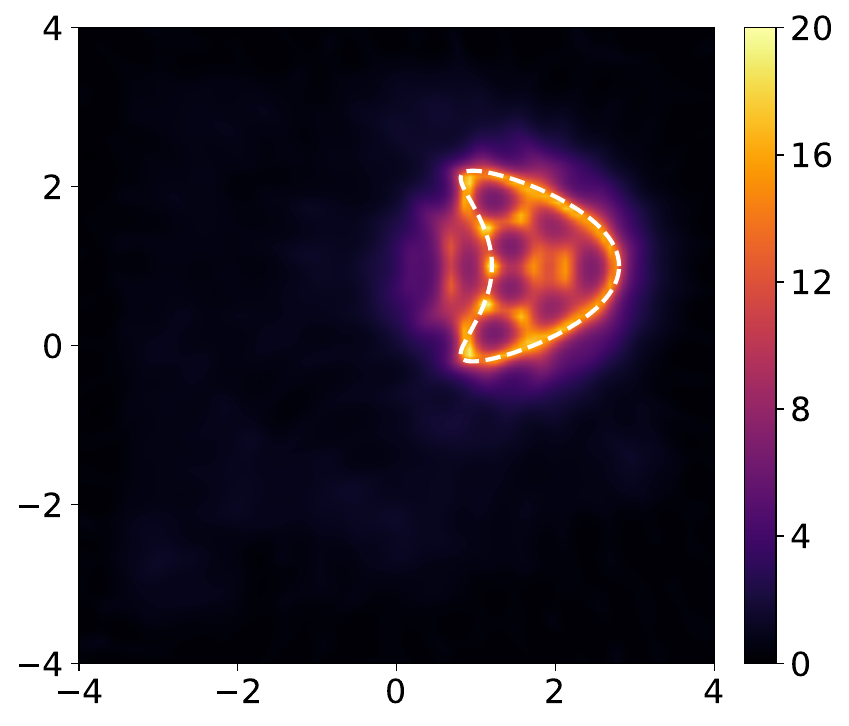}&
 \includegraphics[scale=\scl]{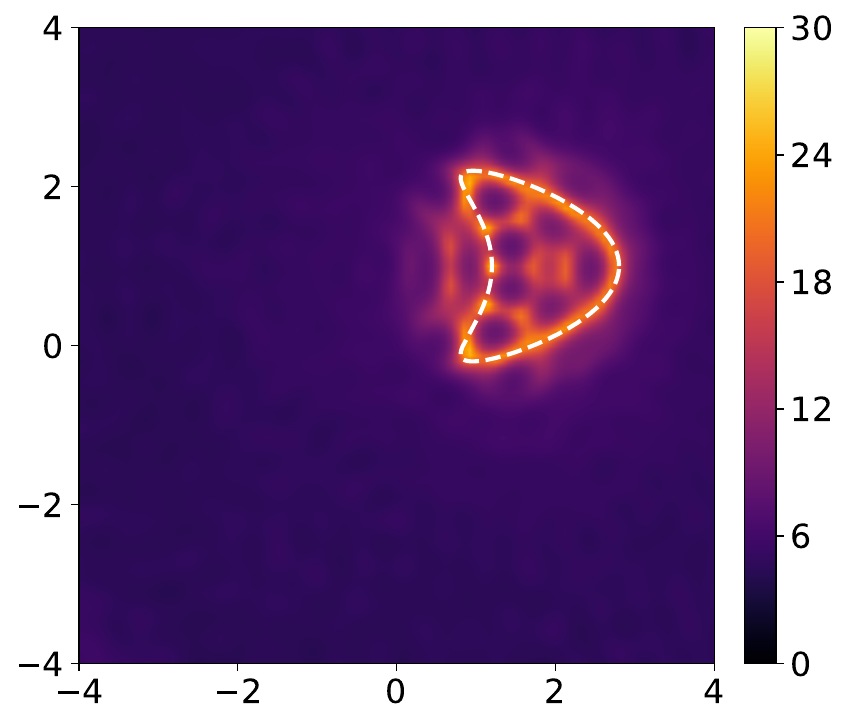} \\

\end{array}  $$
\caption{Reconstructions are shown for a cirle (first row), an ellipse (second row), and a kite (third row). The noise level is fixed to $\eta = 0.1$. The DeepONet accurately recovers both the position and shape of the scatterers. When combined with the LSM, the resulting indicator further improves the geometric reconstruction. Compared to the Morozov-based LSM indicator, the DeepONet-informed LSM yields slightly improved contrast.
}
\label{fig:indicators_single}
\end{figure}

\paragraph{Two scatterers}
We now consider configurations involving two scatterers, which lie outside the training regime, since the network was trained exclusively on single obstacles. Representative results are shown in \cref{fig:indicators_double}. The DeepONet indicators $I_{\theta}(\bs{z})$ are less accurate and contain some noise compared to the single-obstacle case. Nevertheless, the network is able to distinguish the presence of two scatterers and to predict their approximate positions and global sizes. The associated LSM indicator $I(\bs{z})$ significantly improves the reconstruction, yielding results that are comparable to those obtained with the standard Morozov-based LSM indicator $I_{\mathrm{Mor}}(\bs{z})$.

\begin{figure}[!t]
    \def\scl{0.32}
    $$
\begin{array}{ccc}
   I_{\theta}(\bs{z})  \text{: DeepONet } &   I(\bs{z}) \text{: LSM with DeepONet }     & I_{\mathrm{Mor}}(\bs{z}) \text{: LSM with Morozov}  \\
   
   \includegraphics[scale=\scl]{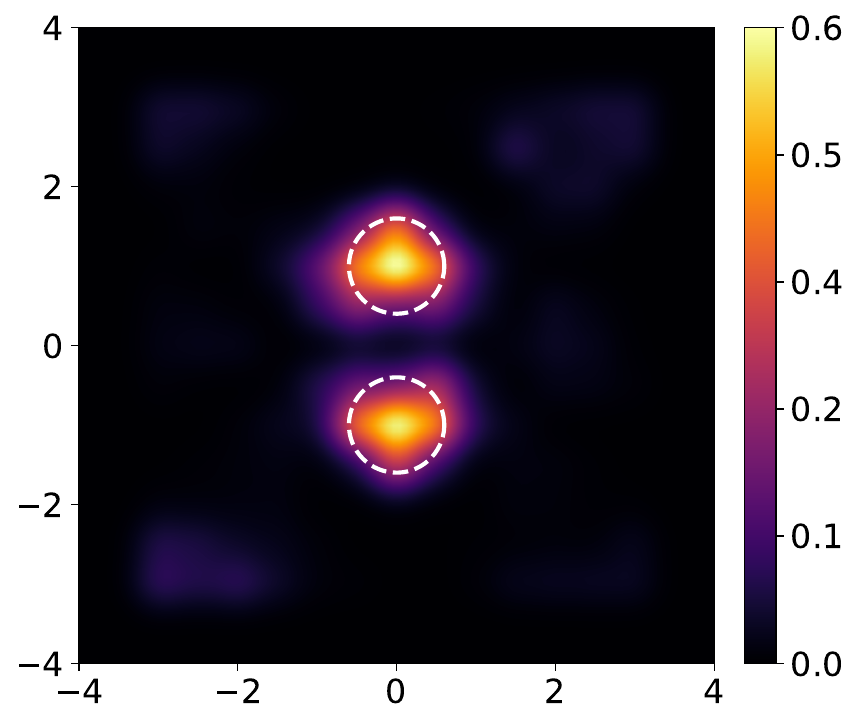}&
 \includegraphics[scale=\scl]{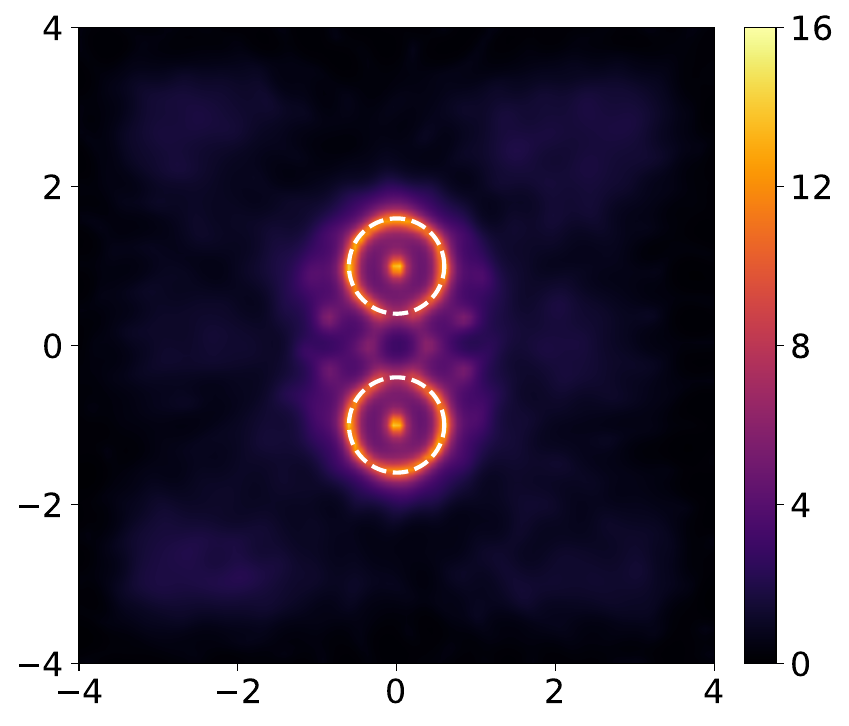}&
 \includegraphics[scale=\scl]{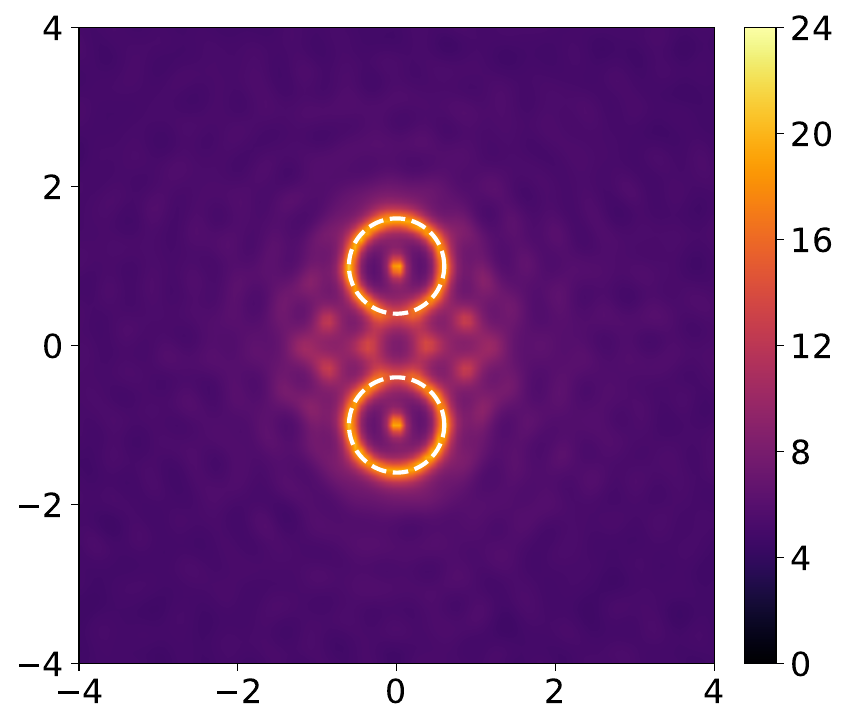} \\

    \includegraphics[scale=\scl]{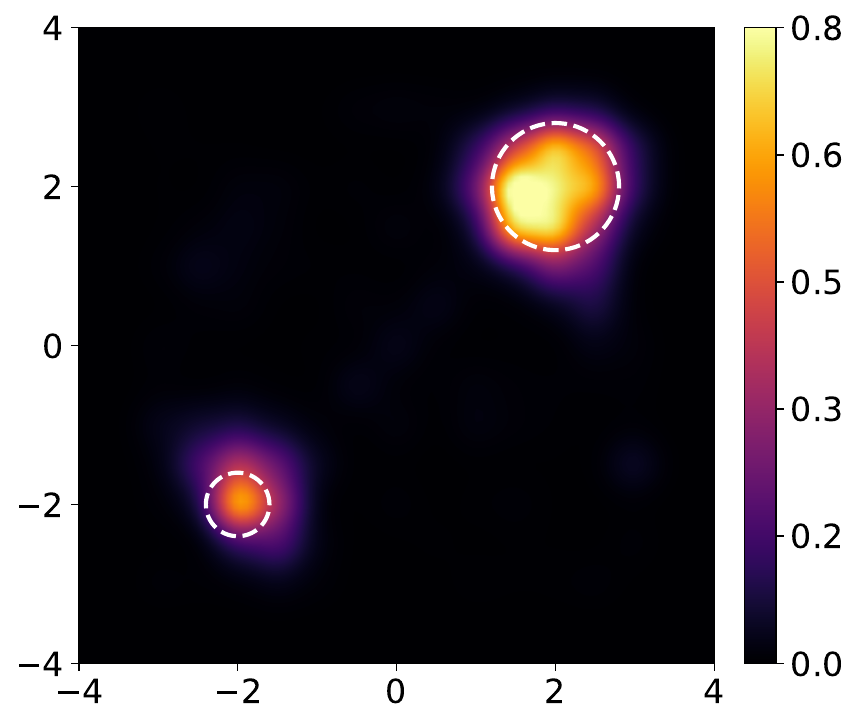}&
 \includegraphics[scale=\scl]{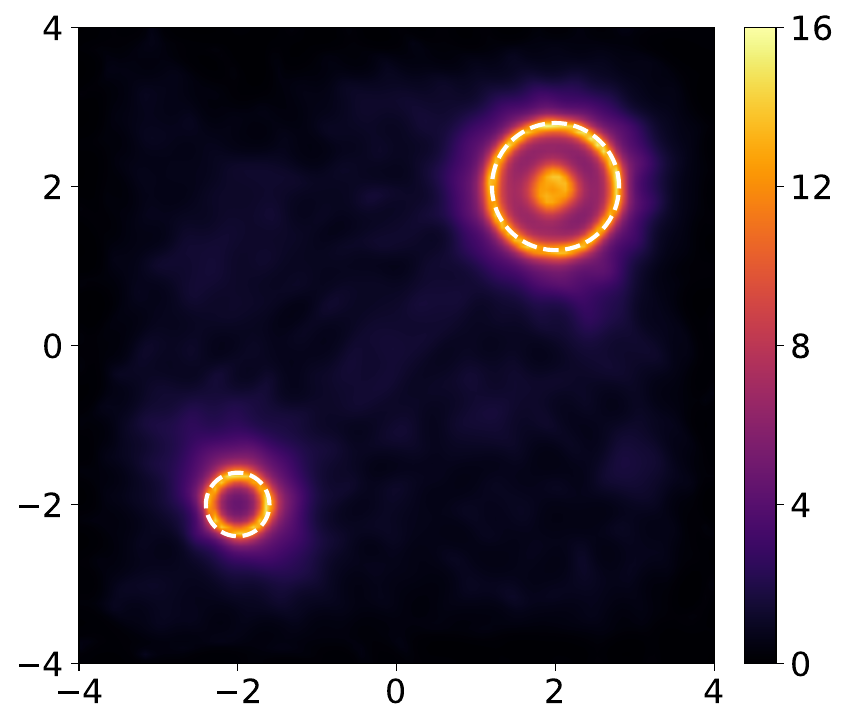}&
 \includegraphics[scale=\scl]{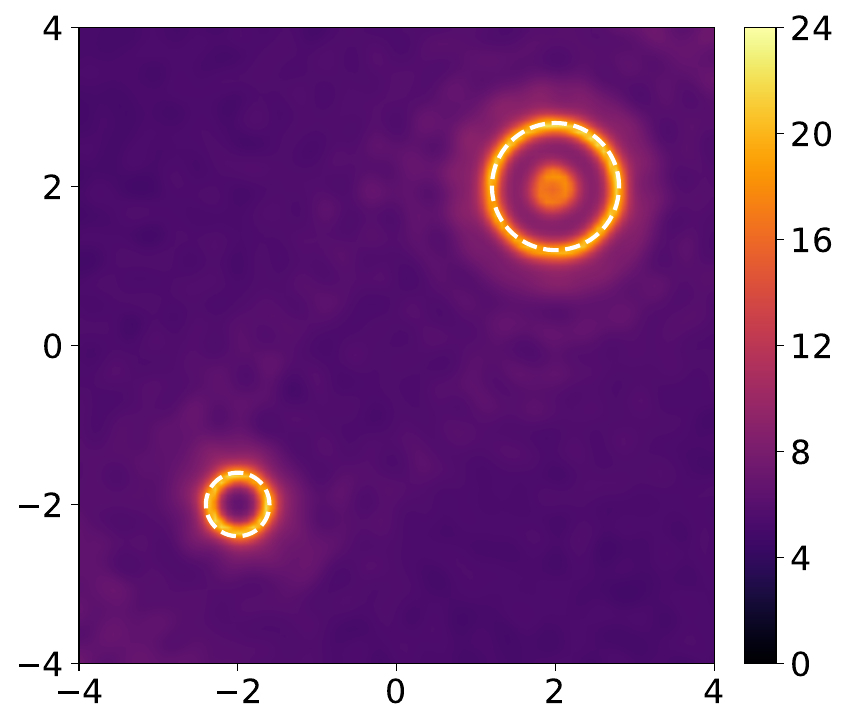} \\

    \includegraphics[scale=\scl]{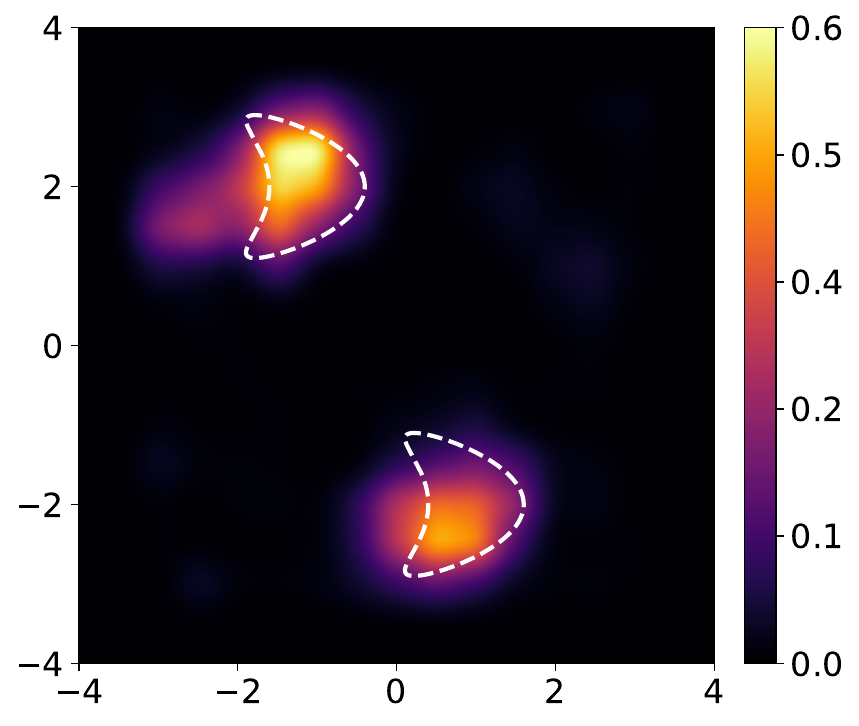}&
 \includegraphics[scale=\scl]{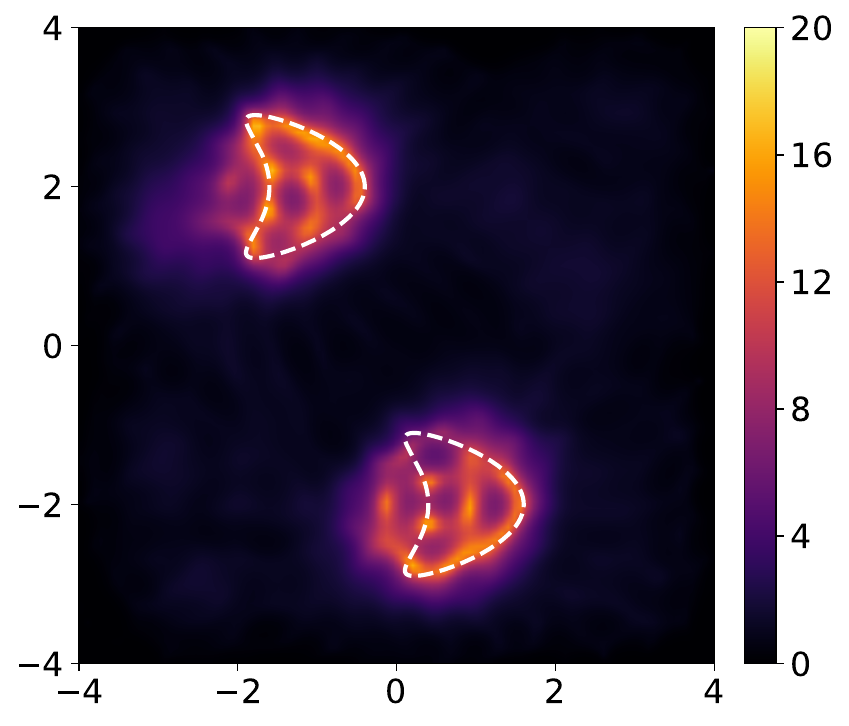}&
 \includegraphics[scale=\scl]{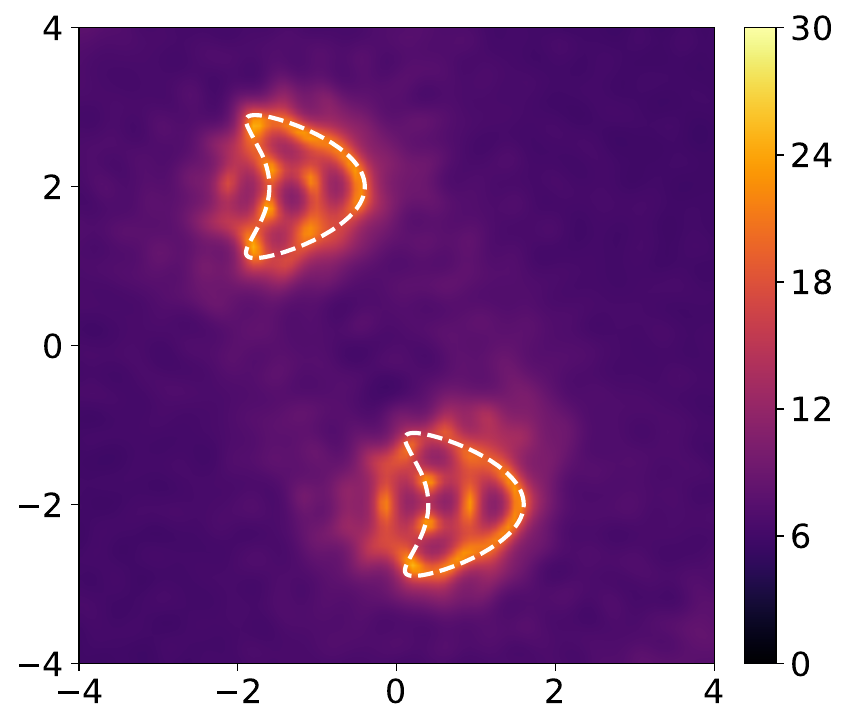} \\

\end{array}  $$
\caption{Reconstructions for different obstacle configurations are shown. The first row corresponds to two circles of radius $0.6$, the second row to two circles with radii $0.5$ and $1$, and the third row to two kites. The noise level is fixed to $\eta = 0.1$. The DeepONet accurately recovers the overall positions and sizes of the obstacles, although their geometry is not accurately resolved. The associated LSM indicator significantly improves the reconstruction, yielding results comparable to those obtained with the Morozov LSM indicator.
}
\label{fig:indicators_double}
\end{figure}

\subsection{Computational time comparison} 

In this section, we compare the computational cost of the standard LSM with that of our neural operator approach, assuming the neural networks have already been trained. 
For a given set of sampling points, we measure the time required to generate the regularization parameters and to compute the indicator function at each point. 
Recall that computing the regularization coefficients with Morozov's discrepancy principle requires solving a nonlinear equation at each sampling point, whereas the neural network approach only involves a forward pass through the network.
In our experiments, we consider $50 \times 50$ far-field matrices. 
The sampling points are chosen on a uniform grid over $\Omega$, and we vary the grid size. 
Tests were conducted on a 2024 MacBook Pro equipped with an Apple M4 Max chip and 36\,GB RAM. 
The results are reported in \cref{tab:times}.

\begin{table}[t!]
\centering
\caption{Computation times for Morozov's discrepancy principle and the DeepONet-based approach, together with the resulting speedup, for different grid sizes. 
The speedup stabilizes around a factor of six as the number of sampling points increases.}
\label{tab:times}
\begin{tabular}{c | c c c}
\toprule
Grid size & Morozov (time) [s] & DeepONet (time) [s] & Speedup \\
\midrule
$10\times 10$  & $2.5\times10^{-2}$ & $1.5\times10^{-2}$ & $1.6$ \\ 
$20\times 20$  & $4.0\times10^{-2}$ & $1.6\times10^{-2}$ & $2.5$ \\ 
$50\times 50$ & $2.5\times10^{-1}$ & $4.0\times10^{-2}$ & $6.4$ \\ 
$100\times 100$ & $6.5\times10^{-1}$ & $9.8\times10^{-2}$ & $6.7$ \\
$200\times 200$ & $2.4\times10^{0}$  & $3.8\times10^{-1}$ & $6.2$ \\
$500\times 500$ & $1.4\times10^{1}$  & $2.4\times10^{0}$  & $6.0$ \\
\bottomrule
\end{tabular}
\end{table}

We observe that, for sufficiently large grid sizes, the speedup stabilizes around a factor of six. 
In other words, the LSM can be performed roughly six times faster when using the neural operator approach. 
However, the networks must first be trained, which entails a fixed computational cost. 
For reference, training the DeepONet required approximately 60\,s on the same machine. 
Consequently, the neural operator method becomes advantageous when applied to a sufficiently large number of test configurations or to sufficiently fine sampling grids. 
The benefit is expected to be even greater in three dimensions, where the number of sampling points increases substantially.


\section{Conclusion}
\label{conclusion}

We proposed a neural operator approach to generate an initial indicator function, which can subsequently be used as a regularization function within LSM after appropriate noise-dependent rescaling. To this end, we introduced a specialized RBF-DeepONet architecture that combines the DeepONet framework with RBF theory, and we conducted an NTK-based analysis to identify parameter choices best suited for achieving optimal performance. In addition, we described a neural network-based process to estimate the noise level from the singular values of the far-field matrix.

The noise prediction process described in \cref{section_noise} seems to be reliable and robust, as it was successfully tested on obstacles of different shapes and for different matrix sizes. Some further theoretical work on singular values of random matrices could be carried out in order to reinforce the justification of our approach.

The neural network indicator performs particularly well for single-obstacle configurations, which fall within the training distribution. In this setting, it accurately recovers the location and overall geometry of the scatterers. When applied to configurations involving multiple obstacles, the indicator becomes less precise, but nevertheless, it remains capable of identifying the correct number of scatterers and provides reasonable estimates of their positions and sizes.
Importantly, when the network-based indicator is used within the LSM framework, the reconstruction quality is improved. The LSM effectively compensates for the inaccuracies of the neural indicator, yielding reconstructions that remain comparable to those obtained using Morozov's discrepancy principle. This highlights the complementary roles of the neural operator and the LSM: while the network provides a fast, data-driven indicator, the LSM restores robustness and consistency.

Once trained, the network-based indicator can be evaluated in essentially constant time for any discretization of the sampling domain, resulting in notable computational savings, particularly for fine sampling grids. This allows a quick, low-cost computation of a qualitative initial indicator. For a more accurate reconstruction, the LSM can then be applied using this initial indicator (appropriately scaled by the predicted noise level) as a regularization function. This approach remains robust while still being significantly faster than determining regularization parameters using Morozov's discrepancy principle.

We believe that our approach shows some potential and that it could be improved or modified to suit particular applications. One possible future work could consist of testing similar approaches on limited aperture setups. Research in this direction could be particularly valuable, especially knowing that the standard LSM struggles in such configurations. The neural operator framework is large enough to experiment with. From changing the depths, widths, the activation functions, or even the architecture of the network, many modifications can be attempted to obtain the best performing model for a certain task. For this purpose, we provide a Python toolbox that can be easily used for experimentation.

\section*{Acknowledgments}
We gratefully acknowledge Hugo Negrel, who first investigated the use of neural networks for inverse scattering in his master's thesis. 
His results motivated the continuation of this research.

\appendix

\section{Asymptotics for Bessel and Hankel functions}
We provide several useful asymptotic formulas for Bessel and Hankel functions. We denote by $J_p$ the $p$-order Bessel function of the first kind and by $H_p$ the $p$-order Hankel function of the first kind.

\subsection{Asymptotics for large arguments}
As $x \to \infty$ with $p \in \mathbb{N}$,
\begin{align}
    H_p(x)=\sqrt{\frac{2}{\pi x}} e^{i(x-p\pi/2-\pi/4)}\left( 1+\mathcal{O} \left( \frac{1}{x} \right) \right).
    \label{asymp_large_argument}
\end{align}
See \cite[eq. (9.2.3)]{abramowitz1964}.
\subsection{Asymptotics for large order}
As $p \to \infty$ with $x \geq 0$,
\begin{align}
    J_p(x)\sim \frac{1}{p!}\left( \frac{x}{2}\right)^p,
    \label{asymp_J_large_order}
\end{align}
and
\begin{align}
  H_p(x) \sim \frac{(p-1)!}{i\pi}  \left(\frac{2}{x} \right)^p .
  \label{asymp_H_large_order}
\end{align}
See \cite[eqs. (9.1.7)-(9.1.9)]{abramowitz1964}.
\section{Analytical formulas for disks}
\label{appendix_a}

This section relies heavily on \cite[Sec.~1.4]{napal2019}, from which we extract several results. The objective is to provide explicit formulas for the far field $u_{\infty}$ when the defect is circular.

Let the defect $D$ be a sound-soft circular obstacle with radius $R>0$ and centered at the origin.
Let us consider a plane incident wave with direction $\hat{d} =(\cos(\mu), \sin(\mu))$:

$$u^i(x):=e^{ik \hat{d}\cdot x}.$$ 
The Jacobi-Anger formula with polar coordinates $(r,\theta)=(|x|,\arg(x))$ yields:
\begin{align*}
   u^i(r,\theta)=  \sum_{n \in \mathbb{Z}} i^n J_{n}(k r) e^{in (\theta-\mu)}.
\end{align*}
 As the scattered field $u^s(r, \theta)$ is a radiating solution to the Helmholtz equation outside the circle, it admits the following expansion: 
\begin{align*}
    u^s(r, \theta)=  \sum_{n \in \mathbb{Z}} a_n H_{n}(kr) e^{i n \theta } \text { for }r>R ,
\end{align*}
where $a_n $ are complex numbers. The condition $u^s=-u^i$ on $\partial D$ yields 
$$a_n=-i^n e^{-in\mu}\frac{J_n(kR)}{H_n(kR)} .$$
 The asymptotic behavior of the Hankel functions \cref{asymp_large_argument} provides the expression for the far-field pattern:
\begin{align*}
u_{\infty}(\theta,\mu)=\sqrt{\frac{2}{k \pi }} e^{-\frac{i \pi}{4}} \sum_{n \in \mathbb{Z}} a_n (-i)^n e^{in\theta} = -\sqrt{\frac{2}{k \pi }} e^{-\frac{i \pi}{4}} \sum_{n \in \mathbb{Z}} \frac{J_n(kR)}{H_n(kR)}e^{in(\theta-\mu)}.
\end{align*}

In the more general case with obstacle centered at $\bs{c}$, this relation becomes 
\begin{align}
u_{\infty}(\theta,\mu)
=
-\sqrt{\frac{2}{k \pi }} e^{-\frac{i \pi}{4}}
\,e^{ik(\bs{\hat d}-\bs{\hat x})\cdot \bs{c}}
\sum_{n \in \mathbb{Z}} \frac{J_n(kR)}{H_n(kR)}\,e^{in(\theta-\mu)} ,
\label{ff_dirichlet_translated}
\end{align}
where $\theta=\arg(\hat x)$ and $\mu=\arg(\hat d)$.


\section{Far-field operator}
We introduce the far-field operator $\mathcal{F}$, which is key in the theoretical study of the LSM (\cite{colton_kress}, Sec. 3.4): 
\begin{align}
    \begin{array}{c}
        \mathcal{F}:  L^2(\mathbb{S}^1) \rightarrow L^2(\mathbb{S}^1)  \\
        \displaystyle
        g \mapsto \int_{\mathbb{S}^1} u_{\infty}(\cdot,\hat{d})\, g(\hat{d}) \, ds(\hat{d}).
    \end{array}
    \label{ff_operator}
\end{align}
Note that this operator is compact since it is an integral operator with a smooth kernel. In \cite{napal2019}, it has been shown that in the case of inhomogeneous circular obstacles, $\mathcal{F}$ is diagonalizable on the canonical basis of $L^2(\mathbb{S}^1)$, and the eigenvalues $\lambda_p$ have the following expressions and asymptotics:
\begin{align}
 \lambda_{p}
 =
 -\sqrt{\frac{8 \pi}{k}} e^{-\frac{i \pi}{4}} \frac{J_p(kR)}{H_p(kR)}
 \underset{p \rightarrow+\infty}{\sim}
 -\sqrt{\frac{8 \pi^3}{k}} e^{\frac{i \pi}{4}}
 \frac{1}{p!(p-1)!}
 \left( \frac{kR}{2} \right)^{2p}. 
 \label{asympotics_dirichlet}
\end{align}
The expression for this asymptotic behavior follows from relations \cref{asymp_J_large_order} and \cref{asymp_H_large_order}.

We provide a useful result regarding the far-field pattern. Considering an obstacle $D$ and its associated far-field pattern $u_{\infty}$, we introduce $u_{\infty,\tau}$ as the far-field pattern of the same obstacle translated by $\tau$, where $\tau \in \mathbb{R}^2$. Then, the following relation holds \cite{colton_kress}:
\begin{align*}
    u_{\infty,\tau}(\hat{x},\hat{d})
    =
    e^{ik \tau \cdot (\hat{d}-\hat{x})}\,
    u_{\infty}(\hat{x},\hat{d}) .
\end{align*}
In particular, this relation implies
\begin{align}
    \mathcal{F}_{\tau}^* \mathcal{F}_{\tau}
    =
    \mathcal{F}^* \mathcal{F},
    \label{invariance_relation}
\end{align}
where $\mathcal{F}$ and $\mathcal{F}_{\tau}$ respectively denote the far-field operator associated with the obstacle $D$ and the far-field operator associated with the translated obstacle. Furthermore, \cref{invariance_relation} implies that $\mathcal{F}$ and $\mathcal{F}_{\tau}$ have the same singular values.

\section{Neural Tangent Kernel Analysis}
We extend standard NTK analysis \cite{jacot2018} in the case of DeepONet with fixed trunk.
\label{appendix:ntk}

\subsection{Neural Tangent Kernel for standard neural networks}
\label{appendix:ntk_classic}
Let
\[
f_\theta : \mathbb{R}^d \to \mathbb{R}^{d'},
\qquad
f_\theta(x)
=
\bigl(f_{\theta,1}(x), \dots, f_{\theta,d'}(x)\bigr)^\top,
\]
and let $\{(x_i, y_i)\}_{i=1}^N$ be a set of training samples with
$y_i \in \mathbb{R}^{d'}$.
The empirical squared loss is given by
\[
\mathcal{L}(\theta)
=
\frac{1}{2}
\sum_{i=1}^N
\| f_\theta(x_i) - y_i \|_2^2
=
\frac{1}{2}
\| y_\theta - y \|_2^2,
\]
where
\[
y_\theta
=
\begin{pmatrix}
f_\theta(x_1)\\
\vdots\\
f_\theta(x_N)
\end{pmatrix}
\in \mathbb{R}^{N d'},
\qquad
y
=
\begin{pmatrix}
y_1\\
\vdots\\
y_N
\end{pmatrix}
\in \mathbb{R}^{N d'}.
\]
Defining the training error $e_\theta = y_\theta - y$ and considering
gradient flow $\theta=\theta(t)$, the evolution of the error satisfies
\[
\frac{d e}{dt}
=
- K(t)\, e,
\]
where the neural tangent kernel $K(t) \in \mathbb{R}^{(N d')\times(N d')}$ is a
block matrix
\[
K(t)
=
\begin{pmatrix}
K_{1,1} & \cdots & K_{1,N}\\
\vdots & \ddots & \vdots\\
K_{N,1} & \cdots & K_{N,N}
\end{pmatrix},
\]
where each block $K_{i,j} \in \mathbb{R}^{d'\times d'}$ is defined by
\[
K_{ij}
=
\nabla_\theta f_\theta(x_i)\,
\nabla_\theta f_\theta(x_j)^\top,
\]
where $\nabla_\theta f_\theta(x)$ is the Jacobian of $f_\theta$ with respect to parameters $\theta$.
In the infinite-width limit, and under appropriate scaling of the parameters,
the NTK converges to a deterministic kernel that remains constant throughout
training. Thus, the training dynamics reduce to a linear system governed
by the spectrum of the kernel matrix $K$ (see \cite{jacot2018}).

\subsection{Neural Tangent Kernel for RBF-DeepONet}
\label{appendix:ntk_DeepONet}
We now specialize the NTK analysis to the DeepONet architecture considered in this work.
The DeepONet is defined as
\[
I_\theta[F](z)
=
u_\theta(F)^\top \Phi_{\epsilon}(z),
\]
where \(u_\theta(F) \in \mathbb{R}^p\) is the output of the branch network and
\[
\Phi_{\epsilon}(z)
=
\bigl(
\phi_{\epsilon}(|z-z_1|),
\dots,
\phi_{\epsilon}(|z-z_p|)
\bigr)^\top
\in \mathbb{R}^p
\]
is a fixed trunk feature map built from radial functions.
Let \(\{F_i\}_{i=1}^{N}\) denote the training inputs and
\(\{I_i^{\mathrm{ref}}\}_{i=1}^{N}\) the corresponding target indicator functions.
Evaluating the output functions at the RBF centers \(\{z_j\}_{j=1}^p\), the (scaled) loss function is defined as
\[
\mathcal{L}(\theta)
=
\frac{1}{2}
\sum_{i=1}^{N}
\sum_{j=1}^{p}
\bigl|
I_\theta[F_i](z_j)
-
I_i^{\mathrm{ref}}(z_j)
\bigr|^2.
\]
We collect all network outputs and targets into the vectors
\[
y_{\theta}
=
\begin{pmatrix}
I_\theta[F_1](z_1)\\
\vdots\\
I_\theta[F_1](z_p)\\
\vdots\\
I_\theta[F_{N}](z_1)\\
\vdots\\
I_\theta[F_{N}](z_p)
\end{pmatrix}
\in \mathbb{R}^{Np},
\qquad
y
=
\begin{pmatrix}
I_1^{\mathrm{ref}}(z_1)\\
\vdots\\
I_1^{\mathrm{ref}}(z_p)\\
\vdots\\
I_{N}^{\mathrm{ref}}(z_1)\\
\vdots\\
I_{N}^{\mathrm{ref}}(z_p)
\end{pmatrix}
\in \mathbb{R}^{Np}.
\]
With this notation, the empirical squared loss can be written compactly as
\[
\mathcal{L}(\theta)
=
\frac{1}{2}\|y_{\theta} - y\|_2^2.
\]
Letting $e = y_{\theta} - y$, the gradient flow dynamics take the form
\[
\frac{d e}{dt} = - K(t)\, e,
\]
where $K(t) \in \mathbb{R}^{(Np)\times(Np)}$ is the DeepONet neural tangent kernel.
Since the trunk network is fixed, differentiation with respect to $\theta$ acts only on the branch network, yielding
\[
\nabla_\theta I_\theta[F](z)
=
\Phi_{\epsilon}(z)^\top
\nabla_\theta u_\theta(F).
\]
The NTK entries then satisfy
\[
K_{\alpha,\alpha'}
=
\Phi_{\epsilon}(z_j)^\top
G_{i,i'}
\Phi_{\epsilon}(z_{j'}),
\qquad
G_{i,i'}
:=
\nabla_\theta u_\theta(F_i)\,
\nabla_\theta u_\theta(F_{i'})^\top
\in \mathbb{R}^{p\times p}.
\]
Define the trunk matrix $P_\epsilon \in \mathbb{R}^{p\times p}$ by
\[
(P_\epsilon)_{ij}
=
\phi_\epsilon\!\left(|z_j-z_i|\right),
\qquad i,j=1,\dots,p,
\]
and define
\[
G :=
\begin{bmatrix}
G_{1,1} & \cdots & G_{1,N} \\
\vdots & \ddots & \vdots \\
G_{N,1} & \cdots & G_{N,N}
\end{bmatrix}
\in \mathbb{R}^{(Np)\times(Np)},
\]
which is the NTK matrix kernel of the branch network $u_\theta$.
One can then write
\[
K
=
\mathbb{P}_{\epsilon}^{\top}
G\,
\mathbb{P}_{\epsilon},
\]
where $\mathbb{P}_{\epsilon}= I_{N}\otimes P_\epsilon \in \mathbb{R}^{(Np)\times(Np)}$.
In the infinite-width limit of the branch network, the neural tangent kernel $K$ remains approximately constant during training, so that the dynamics are again governed by the spectrum of the kernel matrix.

\paragraph{Eigenvalue bound} We first state a general result.

\begin{theorem}[Eigenvalue bound]
\label{thm:eig_bound}
Let $A \in \mathbb{R}^{n\times n}$ be symmetric positive definite and let
\[
B = Q^\top A Q \in \mathbb{R}^{m\times m},
\qquad Q \in \mathbb{R}^{n\times m}.
\]
Then
\[
\lambda_{\min}(A)\,\sigma_{\min}(Q)^2
\;\le\;
\lambda_{\min}(B)
\;\le\;
\lambda_{\max}(B)
\;\le\;
\lambda_{\max}(A)\,\sigma_{\max}(Q)^2.
\]
\end{theorem}

\begin{proof}
Since $B$ is symmetric,
\[
\lambda_{\max}(B)
=
\sup_{x\neq 0}
\frac{x^\top Q^\top A Q x}{\|x\|^2}
=
\sup_{x\neq 0}
\frac{(Qx)^\top A (Qx)}{\|x\|^2}.
\]
For any $x$, set $y=Qx$. Using the Rayleigh quotient bound for $A$ and the
singular value bounds for $Q$,
\[
\frac{(Qx)^\top A (Qx)}{\|x\|^2}
=
\frac{y^\top A y}{\|y\|^2}
\frac{\|Qx\|^2}{\|x\|^2}
\le
\lambda_{\max}(A)\,\sigma_{\max}(Q)^2.
\]
Taking the supremum gives
\[
\lambda_{\max}(B)
\le
\lambda_{\max}(A)\,\sigma_{\max}(Q)^2.
\]
The lower bound follows similarly from
\[
\lambda_{\min}(B)
=
\inf_{x\neq 0}
\frac{(Qx)^\top A (Qx)}{\|x\|^2},
\]
using
\(
\frac{y^\top A y}{\|y\|^2} \ge \lambda_{\min}(A)
\)
and
\(
\frac{\|Qx\|^2}{\|x\|^2} \ge \sigma_{\min}(Q)^2.
\)
\end{proof}

Applying Theorem~\ref{thm:eig_bound} to the factorization
\[
K = \mathbb{P}_{\epsilon}^{\top} G\, \mathbb{P}_{\epsilon},
\]
we obtain
\[
\lambda_{\min}(G)\,\sigma_{\min}(P_{\epsilon})^{2}
\;\le\;
\lambda_{\min}(K)
\;\le\;
\lambda_{\max}(K)
\;\le\;
\lambda_{\max}(G)\,\sigma_{\max}(P_{\epsilon})^{2},
\]
since $\mathbb{P}_{\epsilon}$ and $P_{\epsilon}$ share the same singular values.

\bibliographystyle{siam}
\bibliography{references.bib}

\end{document}